\documentclass[11pt]{amsart}

\usepackage[pagebackref]{hyperref}
\usepackage{amsmath} 
\usepackage{amsthm}
\usepackage{amssymb}
\usepackage[dvipsnames]{xcolor}
\usepackage{tikz-cd}

\usepackage{tikz}
\usetikzlibrary{decorations.markings}
\usetikzlibrary{decorations.pathreplacing}
\usetikzlibrary{arrows,shapes,positioning}

\usepackage{adjustbox}
\usepackage[margin=1.5in]{geometry}
\usepackage{todonotes}
\usepackage{caption}
\usepackage{enumitem}
\usepackage{appendix}
\usepackage{float}
\usepackage{array}
\usepackage{longtable}
\usepackage{tocvsec2}
\usepackage{todonotes}

\hypersetup{citecolor=MidnightBlue,
	colorlinks=true,
	linkcolor=MidnightBlue,
	urlcolor=MidnightBlue}

\newtheorem*{rep@theorem}{\rep@title}
\newcommand{\newreptheorem}[2]{%
	\newenvironment{rep#1}[1]{%
		\def\rep@title{#2 \ref{##1}}%
		\begin{rep@theorem}}%
		{\end{rep@theorem}}}
\makeatother

\newtheorem{theorem}{Theorem}[section]
\newtheorem{lemma}[theorem]{Lemma}
\newtheorem{proposition}[theorem]{Proposition}

\newtheorem{corollary}[theorem]{Corollary}
\newtheorem{theorem*}{Theorem}

\theoremstyle{definition}
\newtheorem{definition}[theorem]{Definition}
\newtheorem{example}[theorem]{Example}

\newtheorem{remark}[theorem]{Remark}

\newreptheorem{theorem}{Theorem}

\DeclareMathOperator{\C}{\mathcal{C}}

\DeclareMathOperator{\Hom}{Hom}
\DeclareMathOperator{\Aut}{Aut}
\DeclareMathOperator{\End}{End}

\DeclareMathOperator{\Vect}{Vec}

\newcommand\VecG{\operatorname{\text{Vec}_G}}
\newcommand\VecGw{\operatorname{\text{Vec}_G^{\omega}}}

\newcommand\ev{\operatorname{ev}}
\newcommand\coev{\operatorname{coev}}
\newcommand{\mycomment}[1]{}

\DeclareMathOperator{\Z}{\mathbb Z}

\title{Diagramatics for cyclic pointed fusion categories}

\makeatletter
\def\l@subsection{\@tocline{2}{0pt}{2.5pc}{5pc}{}}
\def\l@subsubsection{\@tocline{2}{0pt}{2.5pc}{5pc}{}}
\makeatother

\usepackage{tocvsec2}

\begin{document}
	
	\author[A. Czenky]{Agustina Czenky}
	\address{Department of Mathematics, University of Oregon, Eugene, OR 97403, USA}
	\email{aczenky@uoregon.edu}

	\begin{abstract}
		We give a parametrization of cyclic pointed categories associated to the cyclic group of order $n$ in terms of $n$-th roots of unity. We also provide a diagramatic description of these categories by generators and relations, and use it to characterize their $2$-group of automorphisms.
	\end{abstract}
	
		\maketitle 
	\section{Introduction}
	
Let $\mathbf k$ be an algebraically closed field of characteristic zero.  A \emph{fusion category} over $\mathbf k$ is a semisimple rigid tensor category over $\mathbf k$ with finitely many isomorphism classes of simple objects,  finite dimensional hom spaces, and such that the unit object is simple \cite{ENO1}. They are connected to many areas of mathematics and mathematical physics, which makes the problem of classifying them both an important yet difficult task. At the moment, a complete classification of fusion categories seems out of reach as it would yield a classification of finite groups.
Some progress has been made towards the classification of fusion categories of certain classes; this work concerns pointed fusion categories associated to finite cyclic groups.

 An object $X$ in a fusion category $\C$ is called \emph{invertible} if its evaluation and coevaluation morphisms are isomorphisms. A fusion category $\C$ is called \emph{pointed} if every simple object is invertible. Pointed fusion categories are well-understood in terms of finite group data: any pointed fusion category is equivalent to the category $\Vect_G^{\omega}$ of finite dimensional $G$-graded vector spaces, where $G$ is a finite group and $\omega$ is a 3-cocycle on $G$ with coefficients in $\mathbf k$ codifying the associativity isomorphism. Monoidal functors between these categories are also classified: a functor $\Vect_{G_1}^{\omega_1}\to \Vect_{G_2}^{\omega_2}$ is determined by a group homomorphism $f:G_1\to G_2$  satisfying that $\omega_1$ and $f^*\omega_2$ are cohomologous, along with the data of a 2-cocycle on $G_1$ \cite[Section 2.6]{EGNO}.

In this work, the hope is to provide a more efficient way of working with pointed fusion categories associated to finite cyclic groups. To this end, we show in Section \ref{section:application to Vecn} that for the cyclic group $\mathbb Z_n$ of order $n$, pointed fusion categories associated to it have a classification which depends on a unique parameter: an $n$-th root of unity $\zeta\in \mathbf k$. We refer to the corresponding pointed category by  $\Vect_{\mathbb Z_n}^{\zeta}$. We later use this parametrization in Section \ref{sec: proof of equiv} to prove the following theorem.
\begin{theorem*}
	The category $\Vect_{\mathbb Z_n}^{\zeta}$ has a description by generators and relations with diagramatics given in Definition \ref{gens and rels}.
\end{theorem*}

With this description, we get the following universal property: 
\begin{theorem*}
	Tensor functors $\Vect_{\mathbb Z_n}^{\zeta} \to \mathcal C$ are in bijection with   pairs $(X,\lambda)$,  where $X\in \C$ and  $ \lambda: X^{\otimes n}\xrightarrow{\sim} \mathbf 1$ satisfies $\text{id}\otimes \lambda = \zeta \lambda\otimes \text{id}$.
\end{theorem*}
For a proof of this theorem, see Section \ref{section:universal property}. This gives a way of measuring when two cyclic pointed fusion associated to $\mathbb Z_n$  are equivalent, and we give an answer for the number of equivalence classes for a fixed $n$ in Section \ref{sec: 2-group}. 	Lastly, we provide a description of the automorphism 2-group of $\Vect_{\mathbb Z_n}^{\zeta}$ in Section \ref{sec: 2-group}, summarized in the following theorem. 
	
	For $j\in \mathbb Z_n$, denote by $\delta_j$ the simple object  in $\Vect_{\mathbb Z_n}^{\zeta}$ defined by $(\delta_j)_k=0$ if $k\ne j$, and $(\delta_j)_k=\mathbf k.$
	\begin{reptheorem}{maintheorem}
		The 2-group $\Aut_{\otimes}(\Vect_{\mathbb Z_n}^{\zeta})$ has:
		\begin{enumerate}
			\item Objects: monoidal automorphisms $F_{X,\lambda}:\Vect_{\mathbb Z_n}^{\zeta}\xrightarrow{\sim}  \Vect_{\mathbb Z_n}^{\zeta}$, where 
			\begin{itemize}
				\item $X\simeq \delta_j$ for some $j\in \mathbb Z_n^{\times}$ such that $\zeta^{j^2}=\zeta$,
				\item $\lambda$ is an isomorphism $\delta_j^{\otimes n} \xrightarrow{\sim} \mathbf 1,$				
				\item $F_{X, \lambda}$ is the tensor functor obtained from mapping the generator  $\delta_1\mapsto X$ and the canonical isomorphism $\delta_1^{\otimes n}\xrightarrow{\sim }\mathbf 1$ to $\lambda$.
			\end{itemize}
			\item Morphisms: Given two automorphisms $F_{X,\lambda}, F_{Y,\lambda'}: \Vect_{\mathbb Z_n}^{\zeta}\xrightarrow{\sim} \Vect_{\mathbb Z_n}^{\zeta}$, where $X\simeq \delta_j$ and $Y\simeq \delta_k$ for some $j,k\in \mathbb Z_n^{\times}$, we have that
			\begin{align*}
				\Hom_{\Aut_{\otimes}(\Vect_{\mathbb Z_n}^{\zeta})}(F_{X,\lambda}, F_{Y,\lambda'})\simeq
				\begin{cases}
					\Theta &\text{if } j= k,\\
					0 &\text{if } j \ne k,
				\end{cases}
			\end{align*}
			where $\Theta$ denotes the set of $n$-th roots of unity. Moreover, 
			\begin{align*}
				\Aut_{\Aut_{\otimes}(\Vect_{\mathbb Z_n}^{\zeta})}(F_{X,\lambda})\simeq \mathbb Z_n,
			\end{align*}
			for all $F_{X,\lambda}\in \Aut_{\otimes}(\Vect_{\mathbb Z_n}^{\zeta})$.
		\end{enumerate}
	\end{reptheorem}

	\section*{Acknowledgements}
	I want to thank my advisor Victor Ostrik for presenting this problem to me and for his guidance and advice.
	
	\section{Preliminaries}\label{section:preliminaries}
	We work over an algebraically closed field $\mathbf k$ of characteristic zero.   
	We denote the ring of non-negative integers by $\mathbb N$, and by $\mathbb Z_n$  and   $\mathbb Z_n^\times$  the rings of integers mod $n$ and its units, respectively. For a category $\C$, $\Hom_{\C}(X,Y)$ denotes the hom space from $X$ to $Y$. 
	
We recall some useful definitions regarding monoidal and tensor 
categories.	We refer the reader to~\cite{ENO1, EGNO} for a more detailed study of these topics. 
	
	\subsection{Monoidal categories}
		A \emph{monoidal category} is a collection $(\mathcal{C},\otimes,a,\textbf{1},\emph{l},\emph{r})$ where  $\mathcal{C}$ is a category, $\otimes : \mathcal{C} \times \mathcal{C} \to \mathcal{C}$ is a bifunctor, called \emph{tensor product}, $\textbf{1} \in \mathcal{C}$ is the \emph{unit object}, and $a_{X,Y,Z}:(X\otimes Y) \otimes Z \xrightarrow{\sim} X \otimes (Y \otimes Z),\ r_X:X \otimes\textbf{1} \xrightarrow{\sim} X, \ \text{and} \ l_X:\textbf{1} \otimes X \xrightarrow{\sim} X$ are natural isomorphisms, for all $X,Y,Z \in \mathcal{C}$, called \emph{associativity} and \emph{unit isomorphisms}, respectively, such that the diagrams 
	\begin{equation*}
		\begin{tikzcd}[column sep=large]
			((X \otimes Y) \otimes Z) \otimes W \arrow{r}{a_{X \otimes Y,Z,W}} \arrow{d}[left]{a_{X,Y,Z}\otimes id_W}
			&(X \otimes Y) \otimes (Z \otimes W) \arrow{r}{a_{X,Y,Z \otimes W}}
			&X\otimes (Y \otimes(Z\otimes W)) \arrow{d}{id_X \otimes a_{Y,Z,W}}\\
			(X\otimes (Y \otimes Z))\otimes W \arrow{rr}[yshift=-3ex]{a_{X,Y\otimes Z,W}}
			&&X\otimes((Y\otimes Z)\otimes W),
		\end{tikzcd}
	\end{equation*}
	and
	\begin{equation*}
		\begin{tikzcd}[column sep=huge]
			(X \otimes \textbf{1})\otimes Y \arrow{rr}{a_{X,\textbf{1},Y}} \arrow{dr}[xshift=-6ex,yshift=-2ex]{r_X\otimes id_Y}
			& &X \otimes(\textbf{1} \otimes Y) \arrow{dl}{id_X\otimes l_Y} \\
			& X \otimes Y.
		\end{tikzcd}
	\end{equation*}
	commute for all $X,Y, Z,W\in \C$.

A \emph{monoidal subcategory} of $\C$ is a monoidal category $(\mathcal{D},\otimes,a,\textbf{1},\emph{l},\emph{r})$, where $\mathcal{D} \subset \mathcal{C}$ is a full subcategory such that  $\textbf{1}\in \mathcal D$, and $\mathcal D$ is closed by tensor product. For any $X$ in a monoidal category $\C$, we can consider the \emph{monoidal subcategory generated by $X$}, which we will denote by $\langle X \rangle$, with objects $X^{\otimes n}$ for all $n\in \mathbb N$.

Given two monoidal categories $(\mathcal{C},\otimes,a,\textbf{1},\emph{l},\emph{r})$ and $(\mathcal{C'},\otimes',a',\textbf{1}',\emph{l}\ ',\emph{r}')$, a  \emph{monoidal functor}  $\C \to \C'$ is a triple $(F,J,u),$ where $F:\mathcal{C} \to \mathcal{C'}$ is a functor, $u:\textbf{1}' \xrightarrow{\sim} F(\textbf{1})$ is an isomorphism, and $J: \otimes' \circ (F \times F) \to F \circ \otimes$ is a natural isomorphism satisfying certain commutative diagrams, see for example \cite[Definition 2.4.1]{EGNO}.
A monoidal functor $F:\C \to \mathcal D$ is said to be \emph{full} (respectively,  \emph{faithful}) if  $\Hom_{\C}(X,Y) \xrightarrow{F} \Hom_{\mathcal D}(F(X),F(Y))$ is surjective (respectively,  injective)  for all $X, Y \in \C.$ We say $F$ is \emph{essentially surjective} if for all $Y\in \mathcal D$ there exists $X\in \C$ such that $F(X)\simeq Y$;  $F$ is an \emph{equivalence} if it is both essentially surjective and fully-faithful.

\subsubsection{Duals}
	Let $\C$ be a monoidal category and let $X\in \C$.  A \emph{left dual} of $X$ (if it exists) is a triple $(X^*, \ev_X, \coev_X)$ where $X^*$ is an object in $\C$ and $\ev_X:X^* \otimes X \to \textbf{1}$ and $\coev_X:\textbf{1} \to X \otimes X^*$ are morphisms in $\C$ satisfying that the following compositions yield the identity of $X$ and $X^*$, respectively:
{\small 
	\begin{equation*}
		\begin{aligned}
		\begin{tikzcd}[column sep=large]
			X \simeq \textbf{1} \otimes X \arrow{r}{\coev_X\otimes \text{id}_X} &(X\otimes X^*)\otimes X \arrow{r}{a_{X,X^*,X}}  &X \otimes (X^*\otimes X) \arrow{r}{\text{id}_X\otimes \ev_X} &X \otimes \textbf{1} \simeq X ,\\
			X^* \simeq X^*\otimes \textbf{1} \arrow{r}{\text{id}_{X^*}\otimes \coev_X} &X^*\otimes (X\otimes X^*) \arrow{r}{a^{-1}_{X^*,X,X^*}}  &(X^* \otimes X)\otimes X^* \arrow{r}{\ev_{X^*} \otimes \text{id}_{X^*}} &\textbf{1}\otimes X^* \simeq X^*.
		\end{tikzcd}
		\end{aligned}
\end{equation*}}%
We call the morphisms $\ev_X$ and $\coev_X$ \emph{evaluation} and \emph{co-evaluation}, respectively. The definition of a \emph{right dual} for $X$ is analogous. A monoidal category is said to be \emph{rigid} if every object has a left and right dual.

	An object $X$ in a rigid category  $\C$ is said to be \emph{invertible} if its evaluation  and  coevaluation  morphisms are isomorphisms. A fusion category is said to be \emph{pointed} if all simple objects are invertible.
	
	  A \emph{2-group} is a rigid monoidal category in which every object is invertible and all morphisms are isomorphisms. A $2$-group $\mathbf G$ is determined up to monoidal equivalence by a triple $(\pi_0(\mathbf G), \pi_1(\mathbf G), a),$ where $\pi_0(\mathbf G)$ is the group of isomorphism classes of objects in $\mathbf G$, $\pi_1(\mathbf G)$ is the group of automorphisms of the unit object, and $a \in H^3 (\pi_0(\mathbf G), \pi_1(\mathbf G))$ \cite{BL}.

	   Given a monoidal category $\C$, let $\Aut_{\otimes}(\C)$ be the monoidal category with objects the tensor autoequivalences of $\C$,  morphisms monoidal natural isomorphisms, and tensor product given by composition. This makes $\Aut_{\otimes}(\C)$ a $2$-group, with invariants $\pi_0(\Aut_{\otimes}(\C))$ the group of isomorphism classes of automorphisms of $\C$ and $\pi_1(\Aut_{\otimes}(\C))$ the group of monoidal natural isomorphisms of the identity functor. 
	
\subsubsection{Additive envelope}
We call a monoidal category $\mathcal C$ \emph{$\mathbf k$-linear} if for each pair of objects $X,Y\in \mathcal C$ the set $\Hom_{\mathcal C}(X,Y)$ has a \textbf{k}-module structure such that composition of morphisms is \textbf{k}-bilinear.  We say $\mathcal C$ is \emph{additive} if it is \textbf{k}-linear and for every finite sequence of objects $X_1,\dots, X_n $ in $\mathcal C$, there exists their direct sum $X_1 \oplus \dots \oplus X_n$ in $\mathcal C$. 

Given a monoidal category $\C$, we can construct its \emph{additive envelope} $\mathcal{A(C)}$ which has objects
finite formal sums $X_1 \oplus \dots \oplus X_m$ of objects in $\mathcal C$, and morphisms defined as follows. For every $X_1\oplus \dots X_m, Y_1\oplus \dots Y_n$ in $\mathcal{A(C)}$, $$\Hom_{\mathcal{A(C)}}(X_1\oplus \dots X_m, Y_1\oplus \dots \oplus Y_n) := \bigoplus\limits_{i,j} \Hom_{\mathcal C}(X_i, Y_j).$$
		Composition of morphisms is given by matrix multiplication.

\subsection{Tensor categories}

A \emph{tensor category} $\mathcal C$ is a $\textbf{k}$-linear locally finite rigid abelian monoidal category, such that $\End_{\mathcal C}(\textbf 1)\simeq \textbf k$, see e.g. \cite{SR} or \cite{EGNO}. For tensor categories $\mathcal C$ and $\mathcal D$, a \emph{tensor functor} $F:\mathcal C\to \mathcal D$ is a $\textbf{k}$-linear exact and faithful monoidal functor.

	A \emph{fusion category}  is a tensor category which is semisimple and has a finite number of isomorphism classes of simple objects.

\subsection{Pointed categories}
For a group $G$, we can consider the category $\VecG$ of finite-dimensional $G$-graded vector spaces. That is, objects are vector spaces $V$ with a decomposition $V=\bigoplus\limits_{g\in G} V_g$ and morphisms are linear transformations that preserve the grading. Then $\VecG$ is a monoidal category with tensor product given by
\begin{equation*}
(V \otimes W)_g = \bigoplus_{x,y\in G\ : \ xy=g} V_x \otimes W_y ,
\end{equation*}
unit object $\delta_1$ defined by  $(\delta_1)_1=\textbf{k}$ and $(\delta_1)_g=0$ for $g \ne 1$, and associativity given by the identity. 

More generally, given a 3-cocycle $\omega: G\times G\times G \to \textbf{k}^{\times}$, we have the monoidal category $\VecGw$ which modifies the associativity isomorphism in $\VecG$ in the following way.
		For $U,V,W \in \VecGw,$ define $a_{U,V,W}^{\omega} : (U \otimes V) \otimes W \xrightarrow{\sim} U \otimes (V \otimes W)$ as the linear extension of
		\begin{equation*}
			a_{U_g,V_h,W_l}^{\omega}:=\omega(g,h,l) \text{ id}_{U_g \otimes V_h \otimes W_l}: ( U_g\otimes V_h ) \otimes W_l \xrightarrow{\sim} U_g \otimes ( V_h\otimes W_l), 
		\end{equation*}
		for all $g,h,l \in G.$
	
	When $G$ is a finite group, $\VecGw$ is a rigid fusion category. Here the simple objects are given by $\delta_g$ defined by $(\delta_g)_h:=0$ if $h\ne g$, and $(\delta_g)_g:=\mathbf k,$ for all  $g\in G$. The dual of a simple object $\delta_g$ is $(\delta_g)^*=\delta_{g^{-1}}$, for all $g\in G$. In particular, all the simple objects in $\VecGw$ are invertible, and so $\VecGw$ is pointed. 

 It is well known that any pointed fusion category is equivalent to $\VecGw$ for some finite group $G$  and  3-cocycle $\omega$ \cite{ENO1}; cohomologically equivalent $\omega$'s give rise to equivalent categories $\VecGw$. Hence pointed fusion categories are classified in terms of finite group data.  We will assume without loss of generality that the cocycle $\omega$ is normalized, i.e., $\omega(g, 1, h) = 1$, since any 3-cocycle $\omega$ is cohomologous to a normalized one. This will help with computations, as it implies that $l_{\delta_g} = r_{\delta_g} = \text{id}_{\delta_g}$, for all $g\in G$.

	\section{Objects $X$ with an isomorphism $X^{\otimes n}\simeq \mathbf 1$}\label{section:X^n}
	
	Let $\C$ be a tensor category and $X$ an object in $\C,$ and suppose that there exists an isomorphism  $\lambda : X^{\otimes n} \xrightarrow{\sim} \textbf 1$ for some $n\in \mathbb N.$ 
	Since
	\begin{align*}
		\Hom_{\C}(X^{\otimes(n+1)},X)\simeq \Hom_{\C}(X,X)\simeq \mathbf k,
	\end{align*}
	there exists a constant $\xi \in \mathbf k$ such that the diagram
	\begin{equation} \label{main diagram}
		\begin{tikzcd}  
			X^{\otimes (n+1)} \arrow{d}[left]{\lambda  \otimes \text{id} } \arrow{r}{\xi (\text{id} \otimes \lambda)} &  X \otimes \textbf{1} \arrow{d}{r_X} \\
			\textbf{1} \otimes X \arrow{r}{l_X}  &X
		\end{tikzcd}
	\end{equation}
	commutes. 
	A priory, the constant $\xi$ depends on the choice of isomorphism $ X^{\otimes n} \xrightarrow{\sim} \textbf 1.$ We show below this is not the case.
	
	\begin{lemma}\label{lemma: independent of choice of iso}
		The constant $\xi$ in Diagram \eqref{main diagram} is independent of the choice of isomorphism $X^{\otimes n} \xrightarrow{\sim} \textbf{1}$.
	\end{lemma}
	\begin{proof}
		Let $\lambda_1$ and $\lambda_2$ be isomorphisms $X^{\otimes n} \xrightarrow{\sim} \textbf{1},$ and let $\xi_{1}, \xi_{2} \in \mathbf k$ be the respective constants making Diagram \eqref{main diagram} commute. That is, 
		\begin{align*}
			&& l_X(\lambda_1 \otimes \text{id}) = \xi_{1}r_X (\text{id} \otimes \lambda_1)  &&\text{and}
			&& l_X(\lambda_2 \otimes \text{id}) = \xi_{2} r_X(\text{id} \otimes \lambda_2).
		\end{align*}
	We show that $\xi_1=\xi_{2}$. By Schur's Lemma, there exists $\alpha \in \mathbf k$ such that $\lambda_1=\alpha \lambda_2$. Then
		\begin{align*}
			l_X(\lambda_1 \otimes \text{id}) &= \xi_{1}r_X (\text{id} \otimes \lambda_1)=
			\alpha \xi_{1}r_X (\text{id} \otimes \lambda_2) \\
			&=  \alpha \xi_{1} \xi_{2}^{-1}l_X (\lambda_2\otimes \text{id})=\xi_{1} \xi_{2}^{-1}l_X (\lambda_1\otimes \text{id}),
		\end{align*}
		where in the first and third equalities we are using Diagram \eqref{main diagram} for $\lambda_1$ and $\lambda_2$, respectively. It follows that $\xi_1\xi_2^{-1}=1,$ or equivalentely, $\xi_1=\xi_2.$
	\end{proof}
	
	Since $\xi$ does not depend on the isomorphism $X^{\otimes n}\xrightarrow{\sim} \mathbf 1$, we will sometimes refer to it as \emph{the constant associated to $X$}. We have the following property.
	
	\begin{lemma}\label{lemma: xi is root of unity}
		Let $\xi$ be as in Diagram $\eqref{main diagram}.$ Then $\xi$ is an $n^{\text{th}}$  root of unity.
	\end{lemma}
	\begin{proof}
		Fix an isomorphism $\lambda : X^{\otimes n}  \xrightarrow{\sim} \textbf{1}$. We compute $\lambda\otimes \lambda : X^{\otimes n}\otimes X^{\otimes n}\to \textbf{1}$ in two different ways, omitting the left and right unit isomorphisms for simplicity. 
		
		On one hand, we will show by induction on $k$ that $\lambda \otimes \lambda = \xi^{k} \lambda(\text{id}^{\otimes k} \otimes \lambda \otimes \text{id}^{\otimes (n-k)}) $ for all $0\leq k \leq n.$ 
		By functoriality of the tensor product, $$ \lambda \otimes \lambda = \lambda \  (\lambda \otimes \text{id}^{\otimes n} ),$$
		so the statement is true for $k=0$. 
		Assume then that the claim holds for $k<n$. By Diagram \eqref{main diagram}, we get that
		\begin{align*}
			\text{id}^{\otimes k} \otimes \lambda \otimes \text{id}^{\otimes (n-k)}= \xi \text{id}^{\otimes k+1} \otimes \lambda \otimes \text{id}^{\otimes (n-k-1)}
		\end{align*}
		and thus using the inductive hypothesis we conclude
		\begin{align*}
			\lambda \otimes \lambda = \xi^{k} \lambda( \text{id}^{\otimes k} \otimes \lambda \otimes \text{id}^{\otimes (n-k)})=\xi^{k+1} \lambda(\text{id}^{\otimes k+1} \otimes \lambda \otimes \text{id}^{\otimes (n-k-1)}),
		\end{align*}
		as desired. Hence, we have shown that
		\begin{align}\label{2}
			\lambda \otimes \lambda= \xi^{n} \lambda(\text{id}^{\otimes n} \otimes \lambda),
		\end{align}
		for all $n\in \mathbb N$. 
		
		On the other hand, again by functoriality of the tensor product, we know that \begin{equation}\label{3}
			\lambda \otimes \lambda = \lambda (\text{id}^{ \otimes n} \otimes \lambda).
		\end{equation} Combining Equations \eqref{2} and $\eqref{3},$ we obtain $\xi^n=1.$ 
	\end{proof}
	
	For $j\in \mathbb N$, consider now the object $X^{\otimes j}\in \mathcal C$. Given an isomorphism $\lambda : X^{\otimes n} \xrightarrow{\sim} \textbf 1,$ we have an induced isomorphism $\lambda^{\otimes j} : X^{\otimes nj} \xrightarrow{\sim} \textbf 1,$ so it makes sense to wonder about the relationship between $\xi$ and the constant associated to $X^{\otimes j}$. We arrive to the following result.
	
	\begin{lemma}\label{square}
		The constant associated to $X^{\otimes j}$ is  $\xi^{j^2}.$
	\end{lemma}
	
	\begin{proof}
		We want to compute the constant $\zeta$ that makes the diagram 
		\begin{equation*}
			\begin{tikzcd}[column sep=large] 
				X^{\otimes j(n+1)} \arrow{d}[left]{\lambda^{\otimes j}  \otimes \text{id} } \arrow{r}{\zeta (\text{id} \otimes \lambda^{\otimes j})} &  X^{\otimes j} \otimes \textbf{1} \arrow{d}{r_{X^{\otimes j}}} \\
				\textbf{1} \otimes X^{\otimes j} \arrow{r}{l_{X^{\otimes j}}}  &X^{\otimes j}
			\end{tikzcd}
		\end{equation*}
		commute.
		We will omit the left and right unit isomorphisms for simplicity.  
		By Diagram $\eqref{main diagram},$
		\begin{align*}
			\text{id}_{X}^{\otimes j} \otimes \lambda &= \xi \text{id}_X^{\otimes (j-1)}\otimes \lambda \otimes \text{id}_X.
		\end{align*}
		Repeating this step $j-1$ times, we obtain
		\begin{align*}
			\text{id}_{X}^{\otimes j} \otimes \lambda &= \xi^j \ \lambda \otimes \text{id}_X^{\otimes j}.
		\end{align*}
		That is, switching the order of  $\text{id}^{\otimes j}$ with $\lambda$ in the tensor product results in multiplying by $\xi^j$, and so
		\begin{align*}
			\text{id}_{X}^{\otimes j} \otimes \lambda^{\otimes j} 
			&=\xi^{j^2} \lambda^{\otimes j}\otimes \text{id}_X^j
		\end{align*}
		since we do the switching $j$ times. 
		Hence  $\zeta= \xi^{j^2}$, as desired. 
	\end{proof}


	\section{Application to cyclic pointed fusion categories}\label{section:application to Vecn}
	Let $n$ be a natural number and let  $\zeta \in \mathbf k$
	be an $n$-th root of unity. Then $\zeta$ determines a $3$-cocyle $\omega_{\zeta}:\mathbb Z_n\times \mathbb Z_n\times \mathbb Z_n \to \mathbf k^{\times}$ in the following way:
	\begin{align}\label{associativity in VecGw}
		\omega_{\zeta}(i,j,k)= \zeta^{\frac{i(j+k-\overline{j+k})}{n}}, \ \ \text{for all } i,j,k \in \mathbb Z_n,
	\end{align}
	where for an integer $m$ we denote by $\bar m$ the remainder of the division of $m$ by $n$. Moreover, all $3$-cocyles in $\mathbb Z_n$  (modulo coboundaries) are of the form $\omega_{\zeta}$ for some $n$-th root of unity $\zeta$, see for example \cite[Example 2.6.4]{EGNO}.
	We denote by $\Vect_{\mathbb Z_n}^{\zeta}$
	the pointed fusion category corresponding to
	the 3-cocycle $\omega_{\zeta}$ .

	Let $X$ be a generator of $\Vect_{\mathbb Z_n}^{\zeta}.$ Then there exists an isomorphism $X^{\otimes n} \xrightarrow{\sim} \textbf{1}$. Recall that we showed in Section \ref{section:X^n} that the constant associated to $X$ is an $n$-th root of unity, and is independent of the choice of isomorphism $X^{\otimes n} \xrightarrow{\sim} \textbf{1}$. So it makes sense to wonder if it is an invariant in the category $\Vect_{\mathbb Z_n}^{\zeta},$ and what is its relation to $\zeta.$  In fact, we have the following result.

	\begin{lemma}\label{constant for delta_1}
		Consider the generator $\delta_1$ of $\Vect^{\zeta}_{\mathbb Z_n}$. 
		Then $\zeta$ is the unique constant making the diagram 
		\begin{equation*} 
			\begin{tikzcd}  
				\delta_1^{\otimes (n+1)} \arrow{d}[left]{\sim} \arrow{r}{\sim } &  \delta_1 \otimes \mathbf{1} \arrow{d}{r_{\delta_1}} \\
				\mathbf{1} \otimes \delta_1 \arrow{r}{l_{\delta_1}}  &\delta_1
			\end{tikzcd}
		\end{equation*}
		commute.
	\end{lemma}
	
	\begin{proof}
		Choose  $\lambda$ to be the canonical isomorphism $\delta_1^{\otimes n}\xrightarrow{\sim} \textbf{1}.$ Then the constant for which the diagram above commutes is given by the associativity map from $\delta_1^{\otimes n} \otimes \delta_1 \to \delta_1\otimes \delta_1^{\otimes n}.$ That is, it is determined by the following composition 
		of associativity maps:
		{\small
			\begin{align*}
				(\dots ((\delta_1 \otimes \delta_1)\otimes \delta_1)\otimes\dots )\otimes \delta_1 &\xrightarrow{a_{\delta_1,\delta_1,\delta_1}\otimes \text{id}^{\otimes(n-2)}}
				(\dots (\delta_1 \otimes (\delta_1\otimes \delta_1))\otimes\dots )\otimes \delta_1 \\
				&\xrightarrow{a_{\delta_1,\delta_1^{\otimes 2},\delta_1}\otimes \text{id}^{\otimes(n-3)}}
				(\dots (\delta_1 \otimes ((\delta_1\otimes \delta_1)\otimes \delta_1)\otimes\dots )\otimes \delta_1\\
				& \ \ \ \ \vdots\\
				&\xrightarrow{a_{\delta_1,\delta_1^{\otimes (n-1)},\delta_1}}
				\delta_1 \otimes ((\delta_1\otimes \delta_1)\otimes \delta_1)\otimes\dots )\otimes \delta_1).
		\end{align*}}%
		Recall that $a_{\delta_1, \delta_1^{\otimes k}, \delta_1}$ is determined by multiplication by $\omega(1,k,1),$ see Equation \eqref{associativity in VecGw}. Hence the series of compositions above is given by multiplication by the constant
		\begin{align*}
			\prod_{k=1}^{n-1} \omega_{\zeta}(1,k,1) =  \prod_{k=1}^{n-1} \zeta^{\frac{(k+1-\overline{k+1})}{n}}=\zeta,
		\end{align*}
		as desired. \end{proof}
	
	\begin{remark}\label{remark: square in Vec}
		Let $j\in \Z_n$ such that $\gcd(j,n)=1$, so that $\delta_j$ is a generator of $\Vect_{\mathbb Z_n}^{\zeta}.$ By Lemma \ref{square}, we know that the constant associated to $\delta_j$ is $\zeta^{j^2}.$
	\end{remark}

	\section{Diagramatics for $\Vect_{\mathbb Z_n}^{\zeta}$}\label{sec:diagramatics}
	
	In this section we find a description of the category $\Vect_{\mathbb Z_n}^{\zeta}$ by generators and relations, and give the corresponding diagramatics.

	\begin{definition}\label{gens and rels}
		For $\zeta \in \mathbf k$ and $n\in \mathbb N,$  let $\mathcal D_{\zeta,n}$ be the category defined as follows:
		\begin{itemize}
			\item Objects are collections of $k$ side-by-side points: 
			\begin{equation*}
				{
					\begin{tikzpicture}[scale=.2]
						\draw [blue, fill=black] (2,0) circle (0.2cm and 0.2cm);
						\draw [blue, fill=black] (4,0) circle (0.2cm and 0.2cm);
						\node at (6,1) {\tiny $k$};
						\node at (6,0) {$\dots$};
						\draw [blue, fill=black] (8,0) circle (0.2cm and 0.2cm);
					\end{tikzpicture} 
				}\ \ ,
			\end{equation*}
			for all $k\in \mathbb N$, with monoidal structure given by concatenation.
			
			\item Morphisms: consider the set of diagrams $\mathcal S$ constructed by composition (vertical stacking) and tensor product (concatenation)  out of the generating diagrams below:
						\begin{align}\label{generators}
				\scalebox{0.8}{
					\begin{tikzpicture}[scale=.2]
						\draw [very thick] (8,1) to (8,-5);
						\draw [blue, fill=black] (8,-5) circle (0.2cm and 0.2cm);
						\draw [blue, fill=black] (8,1) circle (0.2cm and 0.2cm);
						\node at (12,-2.5) {= $\text{id}_1$,};
					\end{tikzpicture} 
				}
				\scalebox{0.8}{
					\begin{tikzpicture}[scale=.2]
						\draw [very thick] (-8,-6) to [out=-90, looseness=1, in=-90] (4,-6.5);
						\draw [very thick] (-8,-4) to (-8,-6);
						\draw [blue, fill=black] (-8,-4) circle (0.2cm and 0.2cm);
						\draw [very thick] (-6,-4) to (-6,-8.8);
						\draw [blue, fill=black] (-6,-4) circle (0.2cm and 0.2cm);
						\draw [very thick] (-4,-4) to (-4,-9.4);
						\draw [blue, fill=black] (-4,-4) circle (0.2cm and 0.2cm);
						\node at (-2,-6) {\tiny $n$};
						\node at (-2,-7) {$\dots$};
						\draw [very thick] (0,-4) to (0,-9.8);
						\draw [blue, fill=black] (0,-4) circle (0.2cm and 0.2cm);
						\draw [very thick] (2,-4) to (2,-9.2);
						\draw [blue, fill=black] (2,-4) circle (0.2cm and 0.2cm);
						\draw [very thick] (4,-4) to (4,-6.5);
						\draw [blue, fill=black] (4,-4) circle (0.2cm and 0.2cm);
						\node at (7,-7) {\text{ and }};
				\end{tikzpicture} }
				\scalebox{0.8}{
					\begin{tikzpicture}[scale=.2]
						\draw [very thick] (-8,-6) to [out=90, looseness=1, in=90] (4,-6.5);
						\draw [very thick] (-8,-6) to (-8,-8);
						\draw [blue, fill=black] (-8,-8) circle (0.2cm and 0.2cm);
						\draw [very thick] (-6,-3.5) to (-6,-8);
						\draw [blue, fill=black] (-6,-8) circle (0.2cm and 0.2cm);
						\draw [very thick] (-4,-2.8) to (-4,-8);
						\draw [blue, fill=black] (-4,-8) circle (0.2cm and 0.2cm);
						\node at (-2,-5) {\tiny $n$};
						\node at (-2,-6) {$\dots$};
						\draw [very thick] (0,-3) to (0,-8);
						\draw [blue, fill=black] (0,-8) circle (0.2cm and 0.2cm);
						\draw [very thick] (2,-3.7) to (2,-8);
						\draw [blue, fill=black] (2,-8) circle (0.2cm and 0.2cm);
						\draw [very thick] (4,-6.4) to (4,-8);
						\draw [blue, fill=black] (4,-8) circle (0.2cm and 0.2cm);
						\node at (6,-6) {.};
					\end{tikzpicture} 
				}
			\end{align}
		For $k,l\in \mathbb N$, let $\mathcal S_{k,l}$  be the subset of $\mathcal S$ of diagrams going from $k$ points to $l$ points (read from top to bottom). Define $\Hom_{\mathcal D_{\zeta,n}}(k,l)$ to be the $\mathbf k$-linear span of $\mathcal S_{k,l}$, modulo the relations below:
					\item Relations:
			\begin{align}\label{relations 1}
				\scalebox{0.8}{
					\begin{tikzpicture}[scale=.2]
						\draw [very thick] (-8,-6) to [out=-90, looseness=1, in=-90] (4,-6.5);
						\draw [very thick] (-8,-4) to (-8,-6);
						\draw [very thick] (-6,-4) to (-6,-8.8);
						\draw [very thick] (-4,-4) to (-4,-9.4);
						\draw [blue, fill=black] (-8,-5) circle (0.2cm and 0.2cm);
						\draw [blue, fill=black] (-6,-5) circle (0.2cm and 0.2cm);
						\draw [blue, fill=black] (-4,-5) circle (0.2cm and 0.2cm);
						\draw [blue, fill=black] (4,-5) circle (0.2cm and 0.2cm);
						\draw [blue, fill=black] (0,-5) circle (0.2cm and 0.2cm);
						\draw [blue, fill=black] (2,-5) circle (0.2cm and 0.2cm);
						\draw [very thick] (0,-4) to (0,-9.8);
						\draw [very thick] (2,-4) to (2,-9.2);
						\draw [very thick] (4,-4) to (4,-6.5);
						\draw [very thick] (-8,-4) to [out=90, looseness=1, in=90] (4,-4.5);
						\draw [very thick] (-6,-1.3) to (-6,-8);
						\draw [very thick] (-6,-6) to (-6,-8);
						\draw [very thick] (-4,-0.8) to (-4,-8);
						\draw [very thick] (-4,-6) to (-4,-8);
						\node at (-2,-5) {\tiny $n$};
						\node at (-2,-6) {$\dots$};
						\draw [very thick] (0,-0.9) to (0,-8);
						\draw [very thick] (0,-6) to (0,-8);
						\draw [very thick] (2,-1.8) to (2,-8);
						\draw [very thick] (2,-6) to (2,-8);
						\node at (8,-5.5) {= $\text{id}_0$,};
				\end{tikzpicture} } 
			\end{align}
				\begin{align}\label{relations 3}
				\scalebox{0.8}{
					\begin{tikzpicture}[scale=.2]
						\draw [very thick] (-8,1) to [out=-90, looseness=1, in=-90] (4,1.5);
						\draw [very thick] (-8,3) to (-8,1);
						\draw [very thick] (-6,3) to (-6,-1.6);
						\draw [very thick] (-4,3) to (-4,-2.1);
						\draw [blue, fill=black] (-8,3) circle (0.2cm and 0.2cm);
						\draw [blue, fill=black] (-6,3) circle (0.2cm and 0.2cm);
						\draw [blue, fill=black] (-4,3) circle (0.2cm and 0.2cm);
						\node at (-2,1) {\tiny $n$};
						\node at (-2,0) {$\dots$};
						\draw [very thick] (0,3) to (0,-2);
						\draw [very thick] (2,3) to (2,-1.3);
						\draw [very thick] (4,3) to (4,0.9);
						\draw [blue, fill=black] (0,3) circle (0.2cm and 0.2cm);
						\draw [blue, fill=black] (2,3) circle (0.2cm and 0.2cm);
						\draw [blue, fill=black] (4,3) circle (0.2cm and 0.2cm);
						\draw [very thick] (-8,-6) to [out=90, looseness=1, in=90] (4,-6.5);
						\draw [very thick] (-8,-6) to (-8,-8);
						\draw [very thick] (-6,-3.5) to (-6,-8);
						\draw [very thick] (-4,-2.8) to (-4,-8);
						\draw [blue, fill=black] (-8,-8) circle (0.2cm and 0.2cm);
						\draw [blue, fill=black] (-6,-8) circle (0.2cm and 0.2cm);
						\draw [blue, fill=black] (-4,-8) circle (0.2cm and 0.2cm);
						\node at (-2,-5) {\tiny $n$};
						\node at (-2,-6) {$\dots$};
						\draw [very thick] (0,-3) to (0,-8);
						\draw [very thick] (2,-3.7) to (2,-8);
						\draw [very thick] (4,-6) to (4,-8);
						\draw [blue, fill=black] (0,-8) circle (0.2cm and 0.2cm);
						\draw [blue, fill=black] (2,-8) circle (0.2cm and 0.2cm);
						\draw [blue, fill=black] (4,-8) circle (0.2cm and 0.2cm);
						\node at (6,-2.5) {= };
						\draw [very thick] (8,1) to (8,-5);
						\draw [very thick] (10,1) to (10,-5);
						\draw [blue, fill=black] (8,-5) circle (0.2cm and 0.2cm);
						\draw [blue, fill=black] (8,1) circle (0.2cm and 0.2cm);
						\draw [blue, fill=black] (10,-5) circle (0.2cm and 0.2cm);
						\draw [blue, fill=black] (10,1) circle (0.2cm and 0.2cm);
						\node at (12,-2) {\tiny $n$};
						\node at (12,-3) {$\dots$};
						\draw [very thick] (14,1) to (14,-5);
						\draw [blue, fill=black] (14,-5) circle (0.2cm and 0.2cm);
						\draw [blue, fill=black] (14,1) circle (0.2cm and 0.2cm);
						\node at (16,-2.5) {,};
					\end{tikzpicture} 
				} 
			\end{align} and
			\begin{align}\label{relations 2}
				\scalebox{0.8}{
					\begin{tikzpicture}[scale=.2]
						\draw [very thick] (-8,1) to [out=-90, looseness=1, in=-90] (4,1.5);
						\draw [very thick] (-8,3) to (-8,1);
						\draw [very thick] (-6,3) to (-6,-1.6);
						\draw [very thick] (-4,3) to (-4,-2.1);
						\draw [blue, fill=black] (-8,3) circle (0.2cm and 0.2cm);
						\draw [blue, fill=black] (-6,3) circle (0.2cm and 0.2cm);
						\draw [blue, fill=black] (-4,3) circle (0.2cm and 0.2cm);
						\node at (-2,1) {\tiny $n$};
						\node at (-2,0) {$\dots$};
						\draw [very thick] (0,3) to (0,-2);
						\draw [very thick] (2,3) to (2,-1.3);
						\draw [very thick] (4,3) to (4,0.9);
						\draw [blue, fill=black] (0,3) circle (0.2cm and 0.2cm);
						\draw [blue, fill=black] (2,3) circle (0.2cm and 0.2cm);
						\draw [blue, fill=black] (4,3) circle (0.2cm and 0.2cm);
						\draw [very thick] (24,3) to (24,-2.5);
						\draw [very thick] (-10,3) to (-10,-2.5);
						\draw [blue, fill=black] (-10,3) circle (0.2cm and 0.2cm);
						\draw [blue, fill=black] (-10,-2.5) circle (0.2cm and 0.2cm);
						\node at (7,0) {= $\zeta$ };
						\draw [very thick] (10,1) to [out=-90, looseness=1, in=-90] (22,1.5);
						\draw [very thick] (10,3) to (10,1);
						\draw [very thick] (12,3) to (12,-1.6);
						\draw [very thick] (14,3) to (14,-2.1);
						\draw [blue, fill=black] (10,3) circle (0.2cm and 0.2cm);
						\draw [blue, fill=black] (12,3) circle (0.2cm and 0.2cm);
						\draw [blue, fill=black] (14,3) circle (0.2cm and 0.2cm);
						\node at (16,1) {\tiny $n$};
						\node at (16,0) {$\dots$};
						\draw [very thick] (18,3) to (18,-2);
						\draw [very thick] (20,3) to (20,-1.3);
						\draw [very thick] (22,3) to (22,0.9);
						\draw [blue, fill=black] (18,3) circle (0.2cm and 0.2cm);
						\draw [blue, fill=black] (22,3) circle (0.2cm and 0.2cm);
						\draw [blue, fill=black] (20,3) circle (0.2cm and 0.2cm);
						\draw [blue, fill=black] (24,3) circle (0.2cm and 0.2cm);
						\draw [blue, fill=black] (24,-2.5) circle (0.2cm and 0.2cm);
						\node at (26,0) {.};
					\end{tikzpicture} 
				}
			\end{align}
		\end{itemize}
	\end{definition}
	
	\begin{remark}
		Since the non-identity generating maps are in $\Hom_{\mathcal D_{\xi}}(0,n)$ and $\Hom_{\mathcal D_{\xi}}(n,0)$, respectively, it follows that $\Hom_{\mathcal D_{\xi}}(k,l)=0$ if $k$ and $l$ are not congruent modulo $n$. Also, for all $k\in \mathbb N$ we have an isomorphism $k\xrightarrow{\sim} k+n$ given by 
			\begin{align*}
			\scalebox{0.8}{
				\begin{tikzpicture}[scale=.2]
					\draw [very thick] (-8,1) to [out=-90, looseness=1, in=-90] (4,1.5);
					\draw [very thick] (-8,3) to (-8,1);
					\draw [very thick] (-6,3) to (-6,-1.6);
					\draw [very thick] (-4,3) to (-4,-2.1);
					\draw [blue, fill=black] (-8,3) circle (0.2cm and 0.2cm);
					\draw [blue, fill=black] (-6,3) circle (0.2cm and 0.2cm);
					\draw [blue, fill=black] (-4,3) circle (0.2cm and 0.2cm);
					\node at (-2,1) {\tiny $n$};
					\node at (-2,0) {$\dots$};
					\draw [very thick] (0,3) to (0,-2);
					\draw [very thick] (2,3) to (2,-1.3);
					\draw [very thick] (4,3) to (4,0.9);
					\draw [blue, fill=black] (0,3) circle (0.2cm and 0.2cm);
					\draw [blue, fill=black] (2,3) circle (0.2cm and 0.2cm);
					\draw [blue, fill=black] (4,3) circle (0.2cm and 0.2cm);
					\draw [very thick] (-10,3) to (-10,-2.5);
						\node at (-12,1) {\tiny $k$};
					\node at (-12,0) {$\dots$};
									\node at (6,1) {.};
								\draw [very thick] (-14,3) to (-14,-2.5);
					\draw [blue, fill=black] (-10,3) circle (0.2cm and 0.2cm);
					\draw [blue, fill=black] (-10,-2.5) circle (0.2cm and 0.2cm);
					\draw [blue, fill=black] (-14,3) circle (0.2cm and 0.2cm);
					\draw [blue, fill=black] (-14,-2.5) circle (0.2cm and 0.2cm);
				\end{tikzpicture} 
			}
		\end{align*}
	\end{remark}

\begin{example}
	For $n=2$, we have two choices $\zeta=\pm1$. Hence the generating morphisms for $\mathcal D_{\pm 1,n}$ are
		\begin{align*}
		\scalebox{0.8}{
			\begin{tikzpicture}[scale=.2]
				\draw [very thick] (8,1) to (8,-5);
				\draw [blue, fill=black] (8,-5) circle (0.2cm and 0.2cm);
				\draw [blue, fill=black] (8,1) circle (0.2cm and 0.2cm);
				\node at (12,-2.5) {= $\text{id}_1$,};
			\end{tikzpicture} 
		}
		\scalebox{0.8}{
			\begin{tikzpicture}[scale=.2]
				\draw [very thick] (-8,-6) to [out=-90, looseness=1.5, in=-90] (-1,-6.5);
				\draw [very thick] (-8,-4) to (-8,-6);
				\draw [blue, fill=black] (-8,-4) circle (0.2cm and 0.2cm);
				\draw [very thick] (-1,-4) to (-1,-6.5);
				\draw [blue, fill=black] (-1,-4) circle (0.2cm and 0.2cm);
				\node at (2,-7) {\text{ and }};
		\end{tikzpicture} }
		\scalebox{0.8}{
			\begin{tikzpicture}[scale=.2]
				\draw [very thick] (-8,-6) to [out=90, looseness=1.5, in=90] (-1,-6.5);
				\draw [very thick] (-8,-6) to (-8,-8);
				\draw [blue, fill=black] (-8,-8) circle (0.2cm and 0.2cm);
				\draw [very thick] (-1,-6.4) to (-1,-8);
				\draw [blue, fill=black] (-1,-8) circle (0.2cm and 0.2cm);
				\node at (2,-6) {,};
			\end{tikzpicture} 
		}
	\end{align*}
	with relations
		\begin{align*}
		\scalebox{0.8}{
			\begin{tikzpicture}[scale=.2]
				\draw [very thick] (-8,-6) to [out=-90, looseness=1.5, in=-90] (-1,-6.5);
				\draw [very thick] (-8,-4) to (-8,-6);
				\draw [blue, fill=black] (-8,-5) circle (0.2cm and 0.2cm);
				\draw [blue, fill=black] (-1,-5) circle (0.2cm and 0.2cm);
				\draw [very thick] (-1,-4) to (-1,-6.5);
				\draw [very thick] (-8,-4) to [out=90, looseness=1.5, in=90] (-1,-4.5);
				\node at (4,-5.5) {= $\text{id}_0$,};
		\end{tikzpicture} } 
		\scalebox{0.8}{
			\begin{tikzpicture}[scale=.2]
				\draw [very thick] (-8,1) to [out=-90, looseness=1.5, in=-90] (-1,1.5);
				\draw [very thick] (-8,3) to (-8,1);
				\draw [blue, fill=black] (-8,3) circle (0.2cm and 0.2cm);
				\draw [very thick] (-1,3) to (-1,1.2);
				\draw [blue, fill=black] (-1,3) circle (0.2cm and 0.2cm);
				\draw [very thick] (-8,-6) to [out=90, looseness=1.5, in=90] (-1,-6.5);
				\draw [very thick] (-8,-6) to (-8,-8);
				\draw [very thick] (-1,-6.5) to (-1,-8);
				\draw [blue, fill=black] (-1,-8) circle (0.2cm and 0.2cm);
				\draw [blue, fill=black] (-8,-8) circle (0.2cm and 0.2cm);
				\node at (2,-2.5) {= };
				\draw [very thick] (4,1) to (4,-5);
				\draw [very thick] (6,1) to (6,-5);
				\draw [blue, fill=black] (4,-5) circle (0.2cm and 0.2cm);
				\draw [blue, fill=black] (4,1) circle (0.2cm and 0.2cm);
				\draw [blue, fill=black] (6,-5) circle (0.2cm and 0.2cm);
				\draw [blue, fill=black] (6,1) circle (0.2cm and 0.2cm);
				\node at (10,-2.5) {, and};
			\end{tikzpicture} 
		} \hspace{5cm}
	\end{align*}
		\vspace{-3.3cm}
	\begin{align*}
		\scalebox{0.8}{\hspace{8cm}
		\begin{tikzpicture}[scale=.2]
			\draw [very thick] (-8,1) to [out=-90, looseness=1.5, in=-90] (-1,1.5);
			\draw [very thick] (-8,3) to (-8,1);
			\draw [very thick] (-1,3) to (-1,1.2);
			\draw [blue, fill=black] (-1,3) circle (0.2cm and 0.2cm);
			\draw [blue, fill=black] (-8,3) circle (0.2cm and 0.2cm);
			\draw [very thick] (15,3) to (15,-2);
			\draw [very thick] (-10,3) to (-10,-2);
			\draw [blue, fill=black] (-10,3) circle (0.2cm and 0.2cm);
			\draw [blue, fill=black] (-10,-2) circle (0.2cm and 0.2cm);
			\node at (3,0) {= $\pm$ };
			\draw [very thick] (6,1) to [out=-90, looseness=1.5, in=-90] (13,1.5);
			\draw [very thick] (6,3) to (6,1);
			\draw [blue, fill=black] (6,3) circle (0.2cm and 0.2cm);
			\draw [blue, fill=black] (13,3) circle (0.2cm and 0.2cm);
			\draw [very thick] (13,3) to (13,1.1);
			\draw [blue, fill=black] (15,3) circle (0.2cm and 0.2cm);
			\draw [blue, fill=black] (15,-2) circle (0.2cm and 0.2cm);
			\node at (17,0) {.};
		\end{tikzpicture} 
	}
		\end{align*}
\end{example}
		\vspace{0.5cm}
		
Tensor product in $\mathcal D_{\zeta, n}$ is given by concatenation of points and diagrams, respectively. Hence $k\otimes l=k+l$ for all $k,l\in \mathbb N$. Thus the objects in the category $\mathcal D_{\zeta, n}$ are tensor-generated by one point. We show next that $\mathcal D_{\zeta, n}$ is a rigid monoidal category, with the following
	choices.
	\begin{enumerate}
		\item \emph{Associativity:} $\alpha_{k,l,j}:(k\otimes l)\otimes j \to k\otimes (l\otimes j)$ is the identity map $\text{id}_{k+l+j}$ for all $k,l,j\in \mathbb N$,
		\item \emph{Unit:}  given by $0\in \mathbb N,$ and represented by ``no points". The left and right unit isomorphisms $0\otimes k\to k$ and $ k\otimes 0\to k$ 
		are the identity $\text{id}_k$ for all $k\in \mathbb N$.
		\item \emph{Duals}: recall that for all $k\in \mathbb N,$ we have an isomorphism $k\xrightarrow{\sim} k+n$. Hence it is enough to define duals for $k$ such that $0\leq k < n$. In that case, the dual of $k$ is $n-k$ with evaluation and coevaluation maps given by:
		\begin{align*}
		\scalebox{0.8}{
			\begin{tikzpicture}[scale=.2]
				\node at (-14,-5) {$\text{eval}_k :=  \  \tiny \zeta^{-k}$};
				\draw [very thick] (-8,-6) to [out=-90, looseness=1, in=-90] (4,-6.5);
				\draw [very thick] (-8,-4) to (-8,-6);
				\draw [blue, fill=black] (-8,-4) circle (0.2cm and 0.2cm);
				\draw [very thick] (-6,-4) to (-6,-8.8);
				\draw [blue, fill=black] (-6,-4) circle (0.2cm and 0.2cm);
				\draw [very thick] (-4,-4) to (-4,-9.4);
				\draw [blue, fill=black] (-4,-4) circle (0.2cm and 0.2cm);
				\node at (-2,-6) {\tiny $n$};
				\node at (-2,-7) {$\dots$};
				\draw [very thick] (0,-4) to (0,-9.8);
				\draw [blue, fill=black] (0,-4) circle (0.2cm and 0.2cm);
				\draw [very thick] (2,-4) to (2,-9.2);
				\draw [blue, fill=black] (2,-4) circle (0.2cm and 0.2cm);
				\draw [very thick] (4,-4) to (4,-6.5);
				\draw [blue, fill=black] (4,-4) circle (0.2cm and 0.2cm);
					\node at (9,-6) {$:n \to 0$,};
		\end{tikzpicture} }
	\end{align*}
	and
		\begin{align*} 
			\scalebox{0.8}{
				\begin{tikzpicture}[scale=.2]
					\node at (-14,-5) {$\text{coeval}_k := $};
					\draw [very thick] (-8,-6) to [out=90, looseness=1, in=90] (4,-6.5);
					\draw [very thick] (-8,-6) to (-8,-8);
					\draw [blue, fill=black] (-8,-8) circle (0.2cm and 0.2cm);
					\draw [very thick] (-6,-3.5) to (-6,-8);
					\draw [blue, fill=black] (-6,-8) circle (0.2cm and 0.2cm);
					\draw [very thick] (-4,-2.8) to (-4,-8);
					\draw [blue, fill=black] (-4,-8) circle (0.2cm and 0.2cm);
					\node at (-2,-5) {\tiny $n$};
					\node at (-2,-6) {$\dots$};
					\draw [very thick] (0,-3) to (0,-8);
					\draw [blue, fill=black] (0,-8) circle (0.2cm and 0.2cm);
					\draw [very thick] (2,-3.7) to (2,-8);
					\draw [blue, fill=black] (2,-8) circle (0.2cm and 0.2cm);
					\draw [very thick] (4,-6.4) to (4,-8);
					\draw [blue, fill=black] (4,-8) circle (0.2cm and 0.2cm);
						\node at (9,-6) {$: 0 \to n,$};
				\end{tikzpicture} 
			}
		\end{align*}
		respectively.
	\end{enumerate}
	
	The triangle and pentagon axioms are trivially satisfied, as all morphisms involved are identities. We show below that duals are well defined for all $k\in \mathbb N$.  In fact, 
	using Relations \eqref{relations 1} and \eqref{relations 2}, we obtain
	
	\vspace{10cm}
		\begin{align*}
		\scalebox{0.8}{ 
			\begin{tikzpicture}[scale=.2]
				\node at (-12,-5) {$\zeta^{-k}$};
				\draw [very thick] (-8,-6) to [out=90, looseness=1, in=90] (4,-6.5);
				\draw [very thick] (-8,-6) to (-8,-14);
				\draw [very thick] (-6,-3.5) to (-6,-8);
				\draw [very thick] (-4,-2.8) to (-4,-14);
				\node at (-2,-5) {\tiny $n$};
				\node at (-2,-6) {$\dots$};
				\draw [very thick] (0,-3) to (0,-8);
				\draw [very thick] (2,-3.7) to (2,-8);
				\node at (-6,-11) {\tiny $k$};
				\node at (-6,-12) {$\dots$};
				\draw [very thick] (4,-6.4) to (4,-8);
			\end{tikzpicture} 
		}  \hspace{9cm}
	\end{align*}
	\vspace{-2.6cm}
	\begin{align*}
		\scalebox{0.8}	{
			\begin{tikzpicture}[scale=.2]
				\draw [very thick] (-8,-6) to [out=-90, looseness=1, in=-90] (4,-6.5);
				\draw [very thick] (-8,-4) to (-8,-6);	
				\node at (2,-2) {\tiny $k$};
				\node at (2,-3) {$\dots$};
				\draw [very thick] (-6,-4) to (-6,-8.8);
				\draw [very thick] (-4,-4) to (-4,-9.4);
				\node at (-2,-6) {\tiny $n$};
				\node at (-2,-7) {$\dots$};
				\draw [very thick] (0,1) to (0,-9.8);
				\draw [very thick] (2,-4) to (2,-9.2);
				\draw [very thick] (4,1) to (4,-6.5);
		\end{tikzpicture} } \hspace{6cm}
	\end{align*}

	\vspace{-2.6cm}
	\scalebox{0.8}{  \hspace{6cm}
		\begin{tikzpicture}[scale=.2]
			\draw [very thick] (-8,-6) to [out=-90, looseness=1, in=-90] (4,-6.5);
			\draw [very thick] (-8,-4) to (-8,-6);
			\draw [very thick] (-6,-4) to (-6,-8.8);
			\draw [very thick] (-4,-4) to (-4,-9.4);
			\draw [very thick] (0,-4) to (0,-9.8);
			\draw [very thick] (2,-4) to (2,-9.2);
			\draw [very thick] (4,-4) to (4,-6.5);
			\draw [very thick] (-8,-4) to [out=90, looseness=1, in=90] (4,-4.5);
			\draw [very thick] (-6,-1.3) to (-6,-8);
			\draw [very thick] (-6,-6) to (-6,-8);
			\draw [very thick] (-4,-0.8) to (-4,-8);
			\draw [very thick] (-4,-6) to (-4,-8);
			\node at (-14,-5) {$= \zeta^{-k+k}$};
			\node at (-2,-5) {\tiny $n$};
			\node at (-2,-6) {$\dots$};
			\draw [very thick] (0,-0.9) to (0,-8);
			\draw [very thick] (0,-6) to (0,-8);
			\draw [very thick] (2,-1.8) to (2,-8);
			\draw [very thick] (2,-6) to (2,-8);
	\end{tikzpicture} } 
	
	\vspace{-1.5cm}
	\scalebox{0.8}{ \hspace{11cm}
		\begin{tikzpicture}[scale=.2]
			\draw [very thick] (8,1) to (8,-5);
			\draw [blue, fill=black] (8,-5) circle (0.2cm and 0.2cm);
			\draw [blue, fill=black] (8,1) circle (0.2cm and 0.2cm);
			\node at (11,-2) {\tiny $k$};
			\node at (11,-3) {$\dots$};
			\draw [very thick] (14,1) to (14,-5);
			\draw [blue, fill=black] (14,-5) circle (0.2cm and 0.2cm);
			\draw [blue, fill=black] (14,1) circle (0.2cm and 0.2cm);
			\node at (16,-2.5) {=};
		\end{tikzpicture} 
	} 
	
	\vspace{-1cm}
	\scalebox{0.8}{ \hspace{13cm}
		\begin{tikzpicture}[scale=.2]
			\draw [very thick] (8,1) to (8,-5);
			\draw [blue, fill=black] (8,-5) circle (0.2cm and 0.2cm);
			\draw [blue, fill=black] (8,1) circle (0.2cm and 0.2cm);
			\node at (11,-2) {\tiny $k$};
			\node at (11,-3) {$\dots$};
					\node at (16,-3) {,};
			\draw [very thick] (14,1) to (14,-5);
			\draw [blue, fill=black] (14,-5) circle (0.2cm and 0.2cm);
			\draw [blue, fill=black] (14,1) circle (0.2cm and 0.2cm);
		\end{tikzpicture} 
	} 
	\vspace{1cm}

	and
	\begin{align*}
		\scalebox{0.8}{ 
			\begin{tikzpicture}[scale=.2]
					\node at (-18,-5) {$\zeta^{-k}$};
				\draw [very thick] (-8,-6) to [out=90, looseness=1, in=90] (4,-6.5);
				\draw [very thick] (-8,-6) to (-8,-8);
				\draw [very thick] (-6,-3.5) to (-6,-8);
				\draw [very thick] (-4,-2.8) to (-4,-8);
				\node at (-2,-5) {\tiny $n$};
				\node at (-2,-6) {$\dots$};
				\draw [very thick] (0,-3) to (0,-14);
				\draw [very thick] (2,-3.7) to (2,-10);
					\node at (2,-11) {\tiny $n-k$};
				\node at (2,-12) {$\dots$};
				\draw [very thick] (4,-6.4) to (4,-14);
			\end{tikzpicture} 
		}  \hspace{8cm}
	\end{align*}
	\vspace{-2.6cm}
	\begin{align*}
	\scalebox{0.8}	{
			\begin{tikzpicture}[scale=.2]
				\draw [very thick] (-8,-6) to [out=-90, looseness=1, in=-90] (4,-6.5);
				\draw [very thick] (-8,1) to (-8,-6);	
				\node at (-6,-2) {\tiny $n-k$};
				\node at (-6,-3) {$\dots$};
				\draw [very thick] (-6,-4) to (-6,-8.8);
				\draw [very thick] (-4,1) to (-4,-9.4);
				\node at (-2,-6) {\tiny $n$};
				\node at (-2,-7) {$\dots$};
				\draw [very thick] (0,-4) to (0,-9.8);
				\draw [very thick] (2,-4) to (2,-9.2);
				\draw [very thick] (4,-4) to (4,-6.5);
		\end{tikzpicture} } \hspace{8cm}
	\end{align*}

		\vspace{-2.6cm}
		 \scalebox{0.8}{  \hspace{10.2cm}
		\begin{tikzpicture}[scale=.2]
			\draw [very thick] (-8,-6) to [out=-90, looseness=1, in=-90] (4,-6.5);
			\draw [very thick] (-8,-4) to (-8,-6);
			\draw [very thick] (-6,-4) to (-6,-8.8);
			\draw [very thick] (-4,-4) to (-4,-9.4);
			\draw [very thick] (0,-4) to (0,-9.8);
			\draw [very thick] (2,-4) to (2,-9.2);
			\draw [very thick] (4,-4) to (4,-6.5);
			\draw [very thick] (-8,-4) to [out=90, looseness=1, in=90] (4,-4.5);
			\draw [very thick] (-6,-1.3) to (-6,-8);
			\draw [very thick] (-6,-6) to (-6,-8);
			\draw [very thick] (-4,-0.8) to (-4,-8);
			\draw [very thick] (-4,-6) to (-4,-8);
			\node at (-2,-5) {\tiny $n$};
				\node at (6,-6) {=};
			\node at (-2,-6) {$\dots$};
			\draw [very thick] (0,-0.9) to (0,-8);
			\draw [very thick] (0,-6) to (0,-8);
			\draw [very thick] (2,-1.8) to (2,-8);
			\draw [very thick] (2,-6) to (2,-8);
	\end{tikzpicture} } 

		\vspace{-1.5cm}
	\scalebox{0.8}{ \hspace{6cm}
	\begin{tikzpicture}[scale=.2]
		\draw [very thick] (8,1) to (8,-5);
		\draw [blue, fill=black] (8,-5) circle (0.2cm and 0.2cm);
		\draw [blue, fill=black] (8,1) circle (0.2cm and 0.2cm);
		\node at (1,-2) {$= \zeta^{-k-(n-k)}$};
		\node at (11,-2) {\tiny $n-k$};
		\node at (11,-3) {$\dots$};
		\draw [very thick] (14,1) to (14,-5);
		\draw [blue, fill=black] (14,-5) circle (0.2cm and 0.2cm);
		\draw [blue, fill=black] (14,1) circle (0.2cm and 0.2cm);
	\end{tikzpicture} 
} 

	\vspace{-1cm}
\scalebox{0.8}{ \hspace{13.5cm}
	\begin{tikzpicture}[scale=.2]
		\draw [very thick] (8,1) to (8,-5);
		\draw [blue, fill=black] (8,-5) circle (0.2cm and 0.2cm);
		\draw [blue, fill=black] (8,1) circle (0.2cm and 0.2cm);
		\node at (11,-2) {\tiny $n-k$};
		\node at (11,-3) {$\dots$};
		\draw [very thick] (14,1) to (14,-5);
		\draw [blue, fill=black] (14,-5) circle (0.2cm and 0.2cm);
		\draw [blue, fill=black] (14,1) circle (0.2cm and 0.2cm);
			\node at (16,-3) {,};
	\end{tikzpicture} 
} 

	\vspace{0.8cm}
as desired.

\subsection{Proof of the equivalence}\label{sec: proof of equiv}
	We show now that $\mathcal D_{\zeta,n}$ is a diagramatic description of $\Vect_{\mathbb Z_n}^{\zeta}$. That is, we show that $\mathcal D_{\zeta, n}$ is monoidally equivalent to the monoidal subcategory $\langle \delta_1\rangle$ of $\Vect_{\mathbb Z_n}^{\zeta}$ generated by $\delta_1$, whose additive envelope is equivalent to $\Vect_{\Z_n}^{\zeta}$.

	In Section \ref{section:application to Vecn} we proved that $\zeta$ is the constant associated to the generator $\delta_1$. We also know by Lemma \ref{lemma: independent of choice of iso} that this constant is independent of the choice of isomorphism $\delta_1^{\otimes n}\xrightarrow{\sim} \mathbf 1$. Fix then the canonical isomorphism $\lambda:\delta_1^{\otimes n} \xrightarrow{\sim} \mathbf 1$, and let $\mu:\mathbf 1\xrightarrow{\sim}  \delta_1^{\otimes n}$ be its inverse. Define a functor $F :\mathcal D_{\zeta,n}\to \langle \delta_1\rangle$ on generating objects and morphisms by 
	\begin{align}\label{def of F1}
		\begin{aligned}
		F:~\quad \mathcal D_{\zeta, n} & ~\longrightarrow ~\langle  \delta_1\rangle\\
		{
			\begin{tikzpicture}[scale=.2]
				\draw [blue, fill=black] (2,0) circle (0.2cm and 0.2cm);
			\end{tikzpicture} 
		} & ~ \longmapsto ~ \delta_1,\\
		\scalebox{0.8}{
			\begin{tikzpicture}[scale=.2]
				\draw [very thick] (-8,-6) to [out=-90, looseness=1, in=-90] (4,-6.5);
				\draw [very thick] (-8,-4) to (-8,-6);
				\draw [blue, fill=black] (-8,-4) circle (0.2cm and 0.2cm);
				\draw [very thick] (-6,-4) to (-6,-8.8);
				\draw [blue, fill=black] (-6,-4) circle (0.2cm and 0.2cm);
				\draw [very thick] (-4,-4) to (-4,-9.4);
				\draw [blue, fill=black] (-4,-4) circle (0.2cm and 0.2cm);
				\node at (-2,-6) {\tiny $n$};
				\node at (-2,-7) {$\dots$};
				\draw [very thick] (0,-4) to (0,-9.8);
				\draw [blue, fill=black] (0,-4) circle (0.2cm and 0.2cm);
				\draw [very thick] (2,-4) to (2,-9.2);
				\draw [blue, fill=black] (2,-4) circle (0.2cm and 0.2cm);
				\draw [very thick] (4,-4) to (4,-6.5);
				\draw [blue, fill=black] (4,-4) circle (0.2cm and 0.2cm);
		\end{tikzpicture} }
		&	~\longmapsto ~\left( \lambda:\delta_1^{\otimes n} \to \mathbf 1\right),\\
		\scalebox{0.8}{
			\begin{tikzpicture}[scale=.2]
				\draw [very thick] (-8,-6) to [out=90, looseness=1, in=90] (4,-6.5);
				\draw [very thick] (-8,-6) to (-8,-8);
				\draw [blue, fill=black] (-8,-8) circle (0.2cm and 0.2cm);
				\draw [very thick] (-6,-3.5) to (-6,-8);
				\draw [blue, fill=black] (-6,-8) circle (0.2cm and 0.2cm);
				\draw [very thick] (-4,-2.8) to (-4,-8);
				\draw [blue, fill=black] (-4,-8) circle (0.2cm and 0.2cm);
				\node at (-2,-5) {\tiny $n$};
				\node at (-2,-6) {$\dots$};
				\draw [very thick] (0,-3) to (0,-8);
				\draw [blue, fill=black] (0,-8) circle (0.2cm and 0.2cm);
				\draw [very thick] (2,-3.7) to (2,-8);
				\draw [blue, fill=black] (2,-8) circle (0.2cm and 0.2cm);
				\draw [very thick] (4,-6.4) to (4,-8);
				\draw [blue, fill=black] (4,-8) circle (0.2cm and 0.2cm);
			\end{tikzpicture} 
		}
		&	~\longmapsto ~\left(\mu:\mathbf 1 \to \delta_1^{\otimes n}\right),
		\end{aligned}
	\end{align}
	and extend it $\mathbf k$-linearly and monoidally to all morphisms. We want $F$ to be $\mathbf k$-linear monoidal functor that is both essentially surjective and fully-faithful. We prove this in what follows by a series of Lemmas. 

\begin{lemma}
The functor $F$ as above is well-defined. 
\end{lemma}
	\begin{proof}
		It is enough to check that $F$ vanishes  Relations \eqref{relations 1} and \eqref{relations 2}. 	 In fact,
		\begin{align*}
			\mu \circ \lambda = \text{id}_{\mathbf 1} &&\text{and} &&\lambda\circ \mu=\text{id}_{\delta_1^{\otimes n}}
		\end{align*}
		since $\lambda$ and $\mu$ are inverses, and
		\begin{align*}
			\text{id}_{\delta_1}\otimes \lambda^n = \zeta \lambda^n\otimes \text{id}_{\delta_1} 
		\end{align*}
		by Lemma \ref{constant for delta_1}.
	\end{proof}

\begin{lemma}\label{F1 is surjective}
	$F$ is surjective on objects. 
\end{lemma}
\begin{proof}
	This is immediate from the definition of $F$.
\end{proof}

To show $F$ is fully-faithful, we must check that the map
	\begin{align*}
		\Hom_{\mathcal D_{\zeta, n}}(k,l) \xrightarrow{F} \Hom_{\Vect_{\mathbb Z_n}^{\zeta}}(\delta_1^{\otimes k}, \delta_1^{\otimes l})
	\end{align*}
induced by $F$ is bijective for all $k, l\in \mathbb N$. 
	By duality, it is enough to check this for
		\begin{align}\label{eq: surj of F}
		\Hom_{\mathcal D_{\zeta, n}}(k,0) \xrightarrow{F} \Hom_{\Vect_{\mathbb Z_n}^{\zeta}}(\delta_1^{\otimes k}, \mathbf 1)
	\end{align}
	for all $k \in \mathbb N$. 
	
	\begin{lemma}\label{F1 is full}
		The map in Equation $\eqref{eq: surj of F}$ is surjective for all $k\in \mathbb N$.
	\end{lemma}
	\begin{proof}
	When $n$ does not divide $k$, both Hom spaces in Equation \eqref{eq: surj of F} are zero, so there is nothing to show. Suppose then that $n$ divides $k$, that is, that there exists $l\in \mathbb N$ so that $k=ln$. Then $\delta_1^{\otimes k}=\mathbf 1$, so
	\begin{align*}
		\Hom_{\Vect_{\mathbb Z_n}^{\zeta}}(\delta_1^{\otimes k}, \mathbf 1) \simeq \Hom_{\Vect_{\mathbb Z_n}^{\zeta}}(\mathbf 1, \mathbf 1)\simeq \mathbf k
	\end{align*}
and	it is enough to check the map in Equation \eqref{eq: surj of F} is non-zero. In fact, the morphism in $\Hom_{\mathcal D_{\zeta, n}}(k,0)$  given by
		\begin{align*}
		\scalebox{0.8}{
			\begin{tikzpicture}[scale=.2]
				\draw [very thick] (-8,-6) to [out=-90, looseness=1, in=-90] (4,-6.5);
				\draw [very thick] (-8,-4) to (-8,-6);
				\draw [blue, fill=black] (-8,-4) circle (0.2cm and 0.2cm);
				\draw [very thick] (-6,-4) to (-6,-8.8);
				\draw [blue, fill=black] (-6,-4) circle (0.2cm and 0.2cm);
				\draw [very thick] (-4,-4) to (-4,-9.4);
				\draw [blue, fill=black] (-4,-4) circle (0.2cm and 0.2cm);
				\node at (-2,-6) {\tiny $n$};
				\node at (-2,-7) {$\dots$};
				\draw [very thick] (0,-4) to (0,-9.8);
				\draw [blue, fill=black] (0,-4) circle (0.2cm and 0.2cm);
				\draw [very thick] (2,-4) to (2,-9.2);
				\draw [blue, fill=black] (2,-4) circle (0.2cm and 0.2cm);
				\draw [very thick] (4,-4) to (4,-6.5);
				\draw [blue, fill=black] (4,-4) circle (0.2cm and 0.2cm);
		\end{tikzpicture} }
		\scalebox{0.8}{
			\begin{tikzpicture}[scale=.2]
				\draw [very thick] (-8,-6) to [out=-90, looseness=1, in=-90] (4,-6.5);
				\draw [very thick] (-8,-4) to (-8,-6);
				\draw [blue, fill=black] (-8,-4) circle (0.2cm and 0.2cm);
				\draw [very thick] (-6,-4) to (-6,-8.8);
				\draw [blue, fill=black] (-6,-4) circle (0.2cm and 0.2cm);
				\draw [very thick] (-4,-4) to (-4,-9.4);
				\draw [blue, fill=black] (-4,-4) circle (0.2cm and 0.2cm);
				\node at (-2,-6) {\tiny $n$};
				\node at (-2,-7) {$\dots$};
				\draw [very thick] (0,-4) to (0,-9.8);
				\draw [blue, fill=black] (0,-4) circle (0.2cm and 0.2cm);
				\draw [very thick] (2,-4) to (2,-9.2);
				\draw [blue, fill=black] (2,-4) circle (0.2cm and 0.2cm);
				\draw [very thick] (4,-4) to (4,-6.5);
				\draw [blue, fill=black] (4,-4) circle (0.2cm and 0.2cm);
				\node at (-11,-5) {\tiny $l$};
				\node at (-11,-6) {$\dots$};
				\node at (11,-6) {$:ln \to 0$};
		\end{tikzpicture} }
	\end{align*}
 is mapped via $F_1$ to the isomorphism  $\lambda^{\otimes l}:(\delta_1)^{\otimes nl} \xrightarrow{\sim} \mathbf 1.$
		\end{proof}
	
	It remains to check that \eqref{eq: surj of F} is injective. For this, we inspect $\Hom_{\mathcal D_{\zeta,n}}(k,0)$ for all $k\in \mathbb N$.
	
	\begin{lemma}\label{one dim Hom}
		Let $k\in \mathbb N$ be divisible by $n$. Then  $\Hom_{\mathcal D_{\zeta, n}}(k,0)$  is the $\mathbf k-$span of
		\begin{align*}
			\scalebox{0.8}{
				\begin{tikzpicture}[scale=.2]
					\draw [very thick] (-8,-6) to [out=-90, looseness=1, in=-90] (4,-6.5);
					\draw [very thick] (-8,-4) to (-8,-6);
					\draw [blue, fill=black] (-8,-4) circle (0.2cm and 0.2cm);
					\draw [very thick] (-6,-4) to (-6,-8.8);
					\draw [blue, fill=black] (-6,-4) circle (0.2cm and 0.2cm);
					\draw [very thick] (-4,-4) to (-4,-9.4);
					\draw [blue, fill=black] (-4,-4) circle (0.2cm and 0.2cm);
					\node at (-2,-6) {\tiny $n$};
					\node at (-2,-7) {$\dots$};
					\draw [very thick] (0,-4) to (0,-9.8);
					\draw [blue, fill=black] (0,-4) circle (0.2cm and 0.2cm);
					\draw [very thick] (2,-4) to (2,-9.2);
					\draw [blue, fill=black] (2,-4) circle (0.2cm and 0.2cm);
					\draw [very thick] (4,-4) to (4,-6.5);
					\draw [blue, fill=black] (4,-4) circle (0.2cm and 0.2cm);
			\end{tikzpicture} }
			\scalebox{0.8}{
				\begin{tikzpicture}[scale=.2]
					\draw [very thick] (-8,-6) to [out=-90, looseness=1, in=-90] (4,-6.5);
					\draw [very thick] (-8,-4) to (-8,-6);
					\draw [blue, fill=black] (-8,-4) circle (0.2cm and 0.2cm);
					\draw [very thick] (-6,-4) to (-6,-8.8);
					\draw [blue, fill=black] (-6,-4) circle (0.2cm and 0.2cm);
					\draw [very thick] (-4,-4) to (-4,-9.4);
					\draw [blue, fill=black] (-4,-4) circle (0.2cm and 0.2cm);
					\node at (-2,-6) {\tiny $n$};
					\node at (-2,-7) {$\dots$};
					\draw [very thick] (0,-4) to (0,-9.8);
					\draw [blue, fill=black] (0,-4) circle (0.2cm and 0.2cm);
					\draw [very thick] (2,-4) to (2,-9.2);
					\draw [blue, fill=black] (2,-4) circle (0.2cm and 0.2cm);
					\draw [very thick] (4,-4) to (4,-6.5);
					\draw [blue, fill=black] (4,-4) circle (0.2cm and 0.2cm);
					\node at (-11,-5) {\tiny $l$};
					\node at (-11,-6) {$\dots$};
					\node at (11,-6) {$:k \to 0$,};
			\end{tikzpicture} }
		\end{align*}
		where $l\in \mathbb N$ is such that $k=ln$.
		\end{lemma}
		
		\begin{proof}
		Let $f_n$ and $g_n$ denote the generating morphisms $n\to 0$ and $0\to n$  in Equation \eqref{generators}, respectively. We call a map in $\mathcal D_{\zeta, n}$ a \emph{building block} if it is a tensor product of either $f_n$'s and identity maps, or $g_n$'s and identity maps. Then any morphism in $\mathcal D_{\zeta, n}$ is given by vertical stacking of building blocks. 
			
Let $b_1:k_1\to k_2$ and $b_2:k_2\to k_3$ be two building blocks, constructed out of tensor products of  $f_n$'s and $g_n$'s, respectively. Then we can turn the composition $b_2b_1$ into a scalar product of a unique block. In fact, using Relation \eqref{relations 2} we can move any component $f_n$ in $b_1$ left or right to cancel out with a component $g_n$ of $b_2$ using Relation \eqref{relations 1}, until we  run out of either $f_n$'s or $g_n$'s. Hence we reduce $b_2b_1$ to a scalar multiple of a unique building block. Analogously, if $b_1$ and $b_2$, are constructed out of tensor products of $g_n$'s and $f_n$'s, respectively, we can do the same using Relations $\eqref{relations 2}$ and $\eqref{relations 3}$. 

Hence, any map in $\mathcal D_{\zeta, n}$ can be reduced to a map built with only $f_n$'s or $g_n$'s (and identity maps). So if we want a map $k\to 0$, it must be the case that it is built with only $f_n$'s and identities. Let $f:k\to 0$, and let $b_1, \dots, b_l$ be building blocks such that $f=b_l\dots b_1$. Using Relation \eqref{relations 2}, we can write $b_i$ as a scalar multiple of $f_n^{\otimes j_i}\otimes \text{id}^{\otimes n(l-(j_1+\dots j_i))}$ for some $j_i\in \mathbb N$ such that $\sum_{i=1}^l j_i=l$, for all $i=1, \dots, l$. Hence $f$ is a scalar multiple of the composition 
\begin{align*}
 (f_n^{\otimes j_l})\dots (f_n^{\otimes j_2}\otimes \text{id}^{\otimes n(l-(j_1+j_2))})	(f_n^{\otimes j_1}\otimes \text{id}^{\otimes n(l-j_1)})&= f_n^{\otimes j_1}\otimes f_n^{\otimes j_1} \otimes \dots \otimes f_n^{\otimes j_l} \\&=f_n^{\otimes l},
\end{align*}
as desired. 
\end{proof}

\begin{lemma}\label{F1 is faithful}
	The map in Equation \eqref{eq: surj of F} is injective for all $k\in \mathbb N$. 
\end{lemma}

\begin{proof}
	We already know by Lemma \ref{F1 is full} that 	
		$\Hom_{\mathcal D_{\zeta, n}}(k,0) \xrightarrow{F} \Hom_{\Vect_{\mathbb Z_n}^{\zeta}}(\delta_1^{\otimes k}, \mathbf 1)$ is surjective for all $k\in \mathbb N$. By Lemma \ref{one dim Hom} both Hom-spaces are one-dimensional, so it is also injective. 
\end{proof}

\begin{corollary}
The functor	$F: \mathcal D_{\zeta, n}\to \langle \delta_1\rangle$ defined in Equation \eqref{def of F1} is an equivalence.  
\end{corollary}

\begin{proof}
This follows from Lemmas \ref{F1 is full}, \ref{F1 is full} and \ref{F1 is faithful}.
\end{proof}


	\subsection{Universal property of $\Vect_{\mathbb Z_n}^{\zeta}$}
	\label{section:universal property}
	
Let $\C$ be a tensor category. Now that we have a description of $\Vect_{\mathbb Z_n}^{\zeta}$ by generators and relations (see Definition \ref{gens and rels}), to define a tensor functor $F:\Vect_{\mathbb Z_n}^{\zeta}\to \C$ it is enough to make a choice of:
\begin{enumerate}
	\item  an object $X\in \C$, and
	\item  an isomorphism $\lambda: X^{\otimes n}\xrightarrow{\sim} \mathbf 1$ such that $\text{id}\otimes \lambda = \zeta \lambda\otimes \text{id}$.
\end{enumerate}
We refer to the corresponding functor by $F_{X, \lambda}$. That is, $F_{X, \lambda}:\Vect_{\mathbb Z_n}^{\zeta}\to \C$ is the tensor functor obtained from mapping $\delta_1\mapsto X$ and the canonical isomorphism $\delta_1^{\otimes n}\xrightarrow{\sim }\mathbf 1$ to $\lambda$.  We then have a bijection:
	\begin{align*}
&	\begin{Bmatrix}\text{Tensor functors} \\ \Vect_{\mathbb Z_n}^{\zeta} \to \mathcal C  
	\end{Bmatrix}
	\longleftrightarrow 
	\begin{Bmatrix}
		\text{pairs } (X,\lambda), \text{ where } X\in \C \text{ and }   \lambda: X^{\otimes n}\xrightarrow{\sim} \mathbf 1  \\
			\ 	 \text{is such that } \text{id}\otimes \lambda = \zeta \lambda\otimes \text{id}
	\end{Bmatrix}.\\
& \hspace{2.5cm}F_{X, \lambda} \longleftrightarrow (X,\lambda)
\end{align*}

	Formally, we can consider the category $\operatorname{Fun}_{\otimes}(\Vect_{\mathbb Z_n}^{\zeta}, \C)$ with objects tensor functors $\Vect_{\mathbb Z_n}^{\zeta}\to \C$ and morphisms monoidal natural transformations between functors. Then the bijection above can be promoted to an equivalence of categories in the following way. Let $\mathcal C_{\zeta}$ denote the category with:
		\begin{itemize}
			\item Objects are pairs $(X,\lambda),$ where $X$ is an object in $\C$ and $\lambda: X^{\otimes n}\xrightarrow{\sim} \mathbf 1$ satisfies  $\text{id}\otimes \lambda = \zeta \lambda\otimes \text{id}$. 
			\item Maps $(X,\lambda) \to (X', \lambda')$ are morphisms $f:X\to X'$ in $\C$ such that $\lambda'f^{\otimes n}=\lambda$.
		\end{itemize}
		Then we have an equivalence of categories
		\begin{align*}
		\Phi:	\operatorname{Fun}_{\otimes}(\Vect_{\mathbb Z_n}^{\zeta}, \C) &\longrightarrow \mathcal C_{\zeta}\\
			F_{X,\lambda}&\mapsto (X,\lambda)\\
			(\tau : F_{X,\lambda} \to F_{X', \lambda'} ) &\mapsto (\tau_{\delta_1}: X\to X').
		\end{align*}
	\begin{proposition}\label{equivalence of categories}
		The map defined above is in fact an equivalence of categories.
	\end{proposition}	

	\begin{proof}
	To show that $\Phi$ is well defined, we need to prove that  any monoidal natural transformation $\tau : F_{X,\lambda} \to F_{X', \lambda'}$ satisfies $\lambda'\tau_{\delta_1}^{\otimes n}=\lambda$. In fact, this follows from naturality of $\tau$, which gives  the following commutative diagram
	\begin{equation*} 
		\begin{tikzcd}  
			&X^{\otimes n}\arrow{d}[left]{\lambda} \arrow{r}{\tau_{\delta_1}^{\otimes n}} &X'^{\otimes n} \arrow{d}{\lambda'} \\
		   & \mathbf 1 \arrow{r}{\tau_{\mathbf 1}=\text{id}}  & \mathbf 1.
		\end{tikzcd}
	\end{equation*}

	It is clear that $\Phi$ is surjective on objects. It remains to show that $\Phi$ is fully-faithful. For this, we prove that
		\begin{align}\label{bijectivity of Phi}
		\Hom_{\operatorname{Fun}_{\otimes}(\Vect_{\mathbb Z_n}^{\zeta}, \C) }(F_{X,\lambda}, F_{X',\lambda'})\xrightarrow{\Phi} \Hom_{\C_{\zeta}}((X,\lambda), (X', \lambda'))
	\end{align}
	is bijective for every suitable pair $(X,\lambda)$. Let $f:(X,\lambda)\to (X', \lambda')$ be a morphism in $\C_{\zeta}$, and define a monoidal natural transformation $\tau: F_{X,\lambda} \to F_{X', \lambda'}$ by $\tau_{\delta_1}=f$ and extend monoidally everywhere else. It is a quick check that this gives a well defined inverse $\Hom_{\C_{\zeta}}((X,\lambda), (X', \lambda')) \to \Hom_{\operatorname{Fun}_{\otimes}(\Vect_{\mathbb Z_n}^{\zeta}, \C) }(F_{X,\lambda}, F_{X',\lambda'})$ for $\Phi$, and so the statement follows. 
	\end{proof}
	
	\section{Automorphism group of $\Vect_{\mathbb Z_n}^{\zeta}$}\label{sec: 2-group}
	
In this section, we give a description of the 2-group $\Aut_{\otimes}(\Vect_{\mathbb Z_n}^{\zeta})$ of tensor automorphisms of $\Vect_{\mathbb Z_n}^{\zeta}$. 
We start by describing when two such categories are equivalent.

	\begin{theorem}\label{when are two equivalent} Let $\zeta$ and $\xi$ be $n$-th roots of unity. 
		Then 
		\begin{align*}
			\Vect^{\zeta}_{\mathbb Z_n}   \simeq\Vect^{\xi}_{\mathbb Z_n} 
		\end{align*}
		if and only if there exists $j \in \mathbb Z_n^{\times}$ such that $\xi^{j^2}=\zeta$. 
	\end{theorem}
	\begin{proof}
	Recall from Proposition \ref{equivalence of categories} that to define a functor $\Vect_{\mathbb Z_n}^{\zeta} \to \Vect_{\mathbb Z_n}^{\xi}$ it is enough to make a choice of an object $X\in \Vect_{\mathbb Z_n}^{\xi}$ and an isomorphism $\lambda: X^{\otimes n}\xrightarrow{\sim} \mathbf 1$ such that $\text{id}\otimes \lambda = \zeta \lambda \otimes \text{id}$. In particular, this implies that $X$ is invertible, hence $X\simeq  \delta_j$ for some $0\leq j < n.$	Moreover, if we want to have an equivalence, the functor should be surjective and thus $X\simeq \delta_j$ should be a generator.
	The statement then follows from Remark \ref{remark: square in Vec}.
\end{proof}

Now that we have a way of measuring when two cyclic pointed fusion associated to $\mathbb Z_n$  are equivalent, we give an answer for the number of equivalence classes for a fixed $n$. 

	\begin{theorem}
			Let $n \in \mathbb N$ and let $ p_1, \cdots, p_l$ be distinct odd primes such that $n=2^{k_0}p_1^{k_1}\cdots p_l^{k_l}$ for some $k_0, \cdots, k_l \in \mathbb N$. There are $c(n)$ categories of the form $\Vect^{\zeta}_{\mathbb Z_n}$ up to equivalence, where 
			\begin{align*}
			c(n) = \begin{cases} \prod\limits_{i=1}^k (2k_i+1) & \text{if}\  k_0=0,  \\ 2\prod\limits_{i=1}^k (2k_i+1) & \text{if}\ k_0=1,  \\ 
				4\prod\limits_{i=1}^k (2k_i+1) & \text{if}\ k_0=2,  \\ 
				4(k_0-1)\prod\limits_{i=1}^k (2k_i+1) & \text{if}\ k_0\geq 3.   \end{cases} 
		\end{align*}
	\end{theorem}
	\begin{proof}
Let $\theta$ be a primitive $n$-th root of unity. Then any pointed fusion category associated to $\mathbb Z_n$ is of the form $\Vect_{\mathbb Z_n}^{\theta^a}$ for some $a\in \mathbb N$. By Corollary \ref{when are two equivalent}, we know that any two pointed categories $\Vect_{\mathbb Z_n}^{\theta^a}$ and $\Vect_{\mathbb Z_n}^{\theta^b}$ are equivalent if and only if there exists $j\in \mathbb Z_n^{\times}$ such that $\theta^{aj^2}=\theta^b$, or equivalentely, such that $j^2a\equiv b \mod n$.  Consider the action of $\mathbb Z_n^{\times}$ in $\mathbb Z_n$  given by
\begin{align*}
	\mathbb Z_{n}^{\times} &\to \End(\mathbb Z_{n})\\
	j &\mapsto (a \mapsto j^2a).
\end{align*}
That is, the action of left multiplication by squares. 
Then counting equivalence classes of pointed fusion categories associated to $\mathbb Z_n$ reduces to computing the number of orbits of this action. The rest of the proof is dedicated to this computation.

Assume for now that $n=p^k$ for some prime $p$ and $k\in \mathbb N$, and let $x,y\in \mathbb Z_{p^k}$. Then one can get from $x$ to $y$ multiplying by a unit if and only if $x$ and $y$ are divisible by the same power of $p$. So if we are looking at the action of multiplication by squares of units, it is enough to compute the orbits of the action restricted to $\mathcal H_i:=\{ x \in \mathbb Z_{p^k} / x \text{ is divisible by } p^i \text{ but not }  p^{i+1}\}$ for all $i=0, \dots, k-1$.

 Let $x,y\in \mathcal H_i$. Then there exist $x_1, y_1\in \mathbb Z_{p^{k-i}}^{\times}$ such that $x=p^ix_1$ and $y=p^iy_1$. So $x,y$ are in the same orbit if and only if there exists $j\in \mathbb Z_{p^k}^{\times}$ such that 
 \begin{align}
 	p^ix_1 \equiv j^2  p^i y_1\mod{p^k}.
 \end{align}
 That is, 
 \begin{align*}
 	x_1\equiv j^2 y_1 \mod{p^{k-i}}
 \end{align*}
 which is equivalent to
 \begin{align*}
 	y_1y_2^{-1} \equiv j^2 \mod{p^{k-i}}.
 \end{align*}
So if $\mathcal R_{k-i}$ denotes the subgroup of quadratic residues of $\mathbb Z_{p^{k-i}}$, the number of orbits of the action restricted to $\mathcal H_i$ is exactly
 \begin{align*}
 	\Big| \mathbb Z_{p^{k-i}} / \mathcal R_{k-i}\Big|
 \end{align*}
 for all $0\leq i <k.$  When $p$ is odd, $\Big| \mathbb Z_{p^{k-i}} / \mathcal R_{k-i}\Big|=2$ for all $0\leq i<k$, and this together with the fact that $\mathcal H_k=\{0\}$ gives a total number of $2k+1$ orbits. 
	If $p=2,$ then 
		\begin{align*}
			\Big| \mathbb Z_{p^{k-i}} / \mathcal R_{k-i}\Big| = \begin{cases} 1 & \text{if}\  i=k-1  \\ 2 & \text{if}\ i=k-2\\ 4 & \text{if}\ 0\leq i\leq k- 3 .  \end{cases} 
		\end{align*}
		Hence if $k=1$ the action has exactly two orbits, if $k=2$ the action has four orbits and if $k\geq 3$ the action has $4(k-1)$ orbits.

		Lastly, note that for the general case $n=2^{k_0}p_1^{k_1}\cdots p_l^{k_l}$ we have identifications 
				\begin{align*}
			\mathbb Z_m^{\times} \simeq \mathbb Z_{2^{k_0}}^{\times}\oplus \mathbb Z_{p_1^{k_1}}^{\times}\oplus \cdots \oplus \mathbb Z_{p_l^{k_l}}^{\times}
		\end{align*}
		and 
				\begin{align*}
		\End(	\mathbb Z_m) \simeq \End(\mathbb Z_{2^{k_0}})\oplus \End(\mathbb Z_{p_1^{k_1}})\oplus \cdots \oplus \End(\mathbb Z_{p_l^{k_l}}).
		\end{align*}
		Thus if we consider the action of multiplication by a square, the orbit of an element $x=(x_0, \dots, x_l)$ in $\mathbb Z_m$ is identified with the product of the orbits of each $x_i\in \mathbb Z_{p_i^{n_i}}$.  
Hence the number of orbits of the action on $\mathbb Z_m$ is the product of the number of orbits restricted to each of the terms $\mathbb Z_{p_i^{n_i}}$ which we computed previously, and the result follows.
	\end{proof}
	
\begin{remark}
	The first 10 terms of the sequence $c(n)$ are $1,2,3,4,3,6,3,8,5,6$. See sequence A092089 at \href{https://oeis.org/A092089}{OEIS} for more information.
\end{remark}
	
We now wish to give a description of the 2-group $\Aut_{\otimes}(\Vect_{\mathbb Z_n}^{\zeta})$, with objects monoidal automorphisms of $\Vect_{\mathbb Z_n}^{\zeta}$ and morphisms their monoidal natural isomorphisms. 

	It follows from the previous results in this section that monoidal automorphisms of $\Vect_{\mathbb Z_n}^{\zeta}$ are of the form $F_{X, \lambda}$, where $X\simeq \delta_j$ for some $j\in \mathbb Z_n^{\times}$ such that $\zeta^{j^2}=\zeta$, $\lambda$ is any isomorphism $\delta_j^{\otimes n} \xrightarrow{\sim} \mathbf 1$, and $F_{X, \lambda}$ is the tensor functor obtained from mapping $\delta_1\mapsto X$ and the canonical isomorphism $\delta_1^{\otimes n}\xrightarrow{\sim }\mathbf 1$ to $\lambda$, see Section \ref{section:universal property}. 

	
	

For the rest of this section, let	$X\simeq \delta_j$ and $Y\simeq \delta_k$ in $\Vect_{\mathbb Z_n}^{\zeta},$ for $j,k\in \mathbb Z_n^{\times}$ satisfying $\zeta^{j^2}=\zeta=\zeta^{k^2}$, and let $F_{X,\lambda}, F_{Y,\lambda'}$ be two automorphisms $\Vect_{\mathbb Z_n}^{\zeta}\to \Vect_{\mathbb Z_n}^{\zeta}$.
	
	\begin{proposition}
	 If $j\ne  k$, then $\Hom_{\Aut_{\otimes}(\Vect_{\mathbb Z_n}^{\zeta})}(F_{X,\lambda}, F_{Y,\lambda'}) =0.$
	\end{proposition}

	\begin{proof}
		To define a natural transformation $F_{j,\lambda}\to F_{k,\lambda'}$ we need to choose a map $\tau_{\delta_1}: X\to Y$ in $\Vect_{\mathbb Z_n}^{\zeta}$. If $j\not \equiv k \mod n,$ the only possible choice is zero. 
	\end{proof}

	\begin{proposition}
	If $j = k$, then $\Hom_{\Aut_{\otimes}(\Vect_{\mathbb Z_n}^{\zeta})}(F_{X,\lambda}, F_{Y,\lambda'}) \simeq \Theta$ as sets, where $\Theta$ denotes the set of $n$-th roots of unity. 
	\end{proposition}
	
	\begin{proof}
Since $j =k$, we have that $X\simeq Y$. Fix an isomorphism $\Gamma:X\xrightarrow{\sim} Y$. 
By Proposition \ref{equivalence of categories}, defining a monoidal transformation $\tau: F_{X,\lambda}\xrightarrow{\sim}F_{Y,\lambda'}$ is equivalent to choosing a map $\tau_{\delta_1}:X \to Y$ such that $\lambda ' \tau_{\delta_1}^{\otimes n}=\lambda$. Now, since $\lambda, \lambda'\Gamma^{\otimes n} \in \Hom(X^{\otimes n}, \mathbf 1)\simeq \mathbf k,$ there exists $\alpha\in \mathbf k^{\times}$ such that $\lambda = \alpha \lambda'\Gamma^{\otimes n}$. Hence if $\lambda' \tau_{\delta_1}^{\otimes n}=\lambda$, we get that
\begin{align*}
	\lambda' \tau_{\delta_1}^{\otimes n}=\alpha \lambda' \Gamma^{\otimes n},
\end{align*}
and so
\begin{align}\label{eq:tensor product}
	\tau_{\delta_1}^{\otimes n}=\alpha \Gamma^{\otimes n}.
\end{align}
 Since $\Gamma^{-1}\tau_{\delta_1}\in \Hom(X,X)\simeq \mathbf k$, then $\Gamma^{-1}\tau_{\delta_1}$ is a multiple of the identity, and so it follows from Equation \eqref{eq:tensor product} that $$\tau_{\delta_1}=\alpha_n\Gamma,$$ where $\alpha_n$ is some $n$-th root of $\alpha$.

	Note that this does not depend on the choice of isomorphism $\Gamma:X\to Y$. In fact, if $\Gamma':X\to Y$ is another isomorphism, then since $\Hom(X,Y)\simeq \mathbf k$ there exists $a\in \mathbf k'$ such that $\Gamma=a\Gamma'$. Following the same computations for $\Gamma'$, we get that $\tau_{\delta_1}=\beta\Gamma'$ for $\beta$ and $n$-th root of $\alpha a^n$, and so again $\tau_{\delta_1}=\alpha_n \Gamma$ for $\alpha_n$ an $n$-th root of $\alpha$.

Since we now know that our computations do not depend on the choice of $\Gamma$, we may choose $\Gamma$ so that $\lambda= \lambda' \Gamma^{\otimes n}.$ Hence 
\begin{align*}
	\tau_{\delta_1}=\theta_n^k \Gamma, \text{ for } k=0,1 \dots, n-1,
\end{align*}
are all the possible choices for $\tau_{\delta_1}$, where $\theta$ is a primitive $n$-th root of unity.
	\end{proof}

	\begin{remark}
		It follows from the proof that $\Aut_{\otimes}(F_{X,\lambda})\simeq \mathbb Z_n$ as groups.
	\end{remark}

	We summarize the results of this section in the following theorem. 
	
	\begin{theorem}\label{maintheorem}
	The 2-group $\Aut_{\otimes}(\Vect_{\mathbb Z_n}^{\zeta})$ has:
\begin{enumerate}
	\item Objects: monoidal automorphisms $F_{X,\lambda}:\Vect_{\mathbb Z_n}^{\zeta}\xrightarrow{\sim}  \Vect_{\mathbb Z_n}^{\zeta}$, where 
	\begin{itemize}
		\item $X\simeq \delta_j$ for some $j\in \mathbb Z_n^{\times}$ such that $\zeta^{j^2}=\zeta$,
		\item $\lambda$ is an isomorphism $\delta_j^{\otimes n} \xrightarrow{\sim} \mathbf 1,$				
		\item $F_{X, \lambda}$ is the tensor functor obtained from mapping the generator  $\delta_1\mapsto X$ and the canonical isomorphism $\delta_1^{\otimes n}\xrightarrow{\sim }\mathbf 1$ to $\lambda$.
	\end{itemize}
	\item Morphisms: Given two automorphisms $F_{X,\lambda}, F_{Y,\lambda'}: \Vect_{\mathbb Z_n}^{\zeta}\xrightarrow{\sim} \Vect_{\mathbb Z_n}^{\zeta}$, where $X\simeq \delta_j$ and $Y\simeq \delta_k$ for some $j,k\in \mathbb Z_n^{\times}$, we have that
	\begin{align*}
		\Hom_{\Aut_{\otimes}(\Vect_{\mathbb Z_n}^{\zeta})}(F_{X,\lambda}, F_{Y,\lambda'})\simeq
		\begin{cases}
			\Theta &\text{if } j= k,\\
			0 &\text{if } j \ne k,
		\end{cases}
	\end{align*}
	where $\Theta$ denotes the set of $n$-th roots of unity. Moreover, 
	\begin{align*}
		\Aut_{\Aut_{\otimes}(\Vect_{\mathbb Z_n}^{\zeta})}(F_{X,\lambda})\simeq \mathbb Z_n,
	\end{align*}
	for all $F_{X,\lambda}\in \Aut_{\otimes}(\Vect_{\mathbb Z_n}^{\zeta})$.
\end{enumerate}
	\end{theorem}
	
\begin{corollary}
	The $2$-group $\Aut_{\otimes}(\Vect_{\mathbb Z_n}^{\zeta})$ has the following invariants:
	\begin{enumerate}
		\item $\pi_0(\Aut_{\otimes}(\Vect_{\mathbb Z_n}^{\zeta}))=\{j\in \mathbb Z_n / \zeta^{j^2}=\zeta \}< \mathbb Z_n$, and
		\item $\pi_1(\Aut_{\otimes}(\Vect_{\mathbb Z_n}^{\zeta})) \simeq \mathbb Z_n$. 
	\end{enumerate}
\end{corollary}

	   \color{black}
	
	\color{black}
	\bibliographystyle{alpha}
	\bibliography{document} 
	
\end{document}